\documentclass[a4paper,11pt,final]{article}

\usepackage{amssymb,amsmath,amsthm}
\usepackage{enumerate,esvect}
\usepackage{a4wide}
\usepackage{hyperref}
\usepackage{graphicx}

\numberwithin{equation}{section}
\allowdisplaybreaks[4]

\newtheorem{proposition}{Proposition}[section]
\newtheorem{theorem}[proposition]{Theorem}
\newtheorem{lemma}[proposition]{Lemma}
\newtheorem{corollary}[proposition]{Corollary}
\newtheorem{definition}[proposition]{Definition}

\renewenvironment{proof}{\smallskip\noindent\emph{\textbf{Proof.}}%
  \hspace{1pt}}{\hspace{-5pt}{\nobreak\quad\nobreak\hfill\nobreak%
    $\square$\vspace{2pt}\par}\smallskip\goodbreak}

\newenvironment{proofof}[1]{\smallskip\noindent{\textbf{Proof~of~#1.}}%
  \hspace{1pt}}{\hspace{-5pt}{\nobreak\quad\nobreak\hfill\nobreak%
    $\square$\vspace{2pt}\par}\smallskip\goodbreak}

\newcommand{\C}[1]{\mathbf{C}^{#1}}
\newcommand{\Cc}[1]{\mathbf{C}_c^{#1}}
\newcommand{\modulo}[1]{{\left|#1\right|}}
\newcommand{\norma}[1]{{\left\|#1\right\|}}
\newcommand{\reali}{{\mathbb{R}}}
\newcommand{\naturali}{{\mathbb{N}}}
\newcommand{\BV}{\mathbf{BV}}

\renewcommand{\epsilon}{\varepsilon}
\renewcommand{\phi}{\varphi}
\renewcommand{\L}[1]{{\mathbf{L}^#1}}
\newcommand{\W}[2]{{\mathbf{W}^{#1,#2}}}

\newcommand{\Lloc}[1]{{\mathbf{L}_{loc}^{#1}}}
\newcommand{\tv}{\mathinner{\rm TV}}
\newcommand{\caratt}[1]{{\chi_{\strut#1}}}

\renewcommand{\div}{\mathinner{\rm div}}

\newcommand{\pt}{\partial}

\newcommand{\conv}{\ast}
\newcommand{\convn}{\star}
\newcommand{\direct}{V}
\renewcommand{\d}[1]{\mathinner{\mathrm{d}{#1}}}
\newcommand{\D}{{\mathrm{D}}}

\renewcommand{\vec}[1]{\vv{#1}}

\begin{document}

\title{Nonlocal Crowd Dynamics Models\\ for Several Populations}

\author{Rinaldo M.~Colombo$^1$ \and Magali L\'ecureux-Mercier$^2$}

\footnotetext[1]{Department of Mathematics, Brescia University, Via
  Branze 38, 25133 Brescia, Italy}

\footnotetext[2]{Technion, Israel Institute of Technology, Amado Building,
 32000 Haifa, Israel}

\maketitle

\begin{abstract}
  \noindent This paper develops the basic analytical theory related to
  some recently introduced crowd dynamics models. Where well posedness
  was known only locally in time, it is here extended to all of
  $\reali^+$. The results on the stability with respect to the
  equations are improved. Moreover, here the case of several
  populations is considered, obtaining the well posedness of systems
  of multi-D non-local conservation laws. The basic analytical tools
  are provided by the classical Kru\v zkov theory of scalar
  conservation laws in several space dimensions.

  \medskip


  \medskip


\end{abstract}

\section{Introduction}
\label{sec:Intro}

\par From a macroscopic point of view, a crowd can be described
through its density $\rho$, with $\rho \in \L1 (\reali^2; \reali)$,
and assuming that $\rho$ satisfies a continuity equation of the form
\begin{equation}
  \label{eq:SCL}
  \partial_t \rho + \div (\rho \, \direct) = 0 \,.
\end{equation}
The vector $\direct$ is in general a function of the space coordinate
$x \in \reali^2$ and of the density $\rho$. The latter dependence may
well be also of functional type since, in general, it is realistic to
assume that $\direct$ depends on $\rho$ through some sort of weighted
space average of $\rho$.

\smallskip

A first example of a model of this kind was presented
in~\cite[Section~4]{ColomboHertyMercier}. There, it is assumed that
pedestrians follow prescribed paths but adjust their speeds to the
density they evaluate near to their positions. This amounts to
postulate a speed law of the form:
\begin{equation}
  \label{eq:Ped}
  \direct =  v( \rho \conv \eta) \, \vec v (x) \,.
\end{equation}
The integral curves of the vector field $\vec v$ are the trajectories
followed by the pedestrians. For instance, $\vec v (x)$ is the unit
tangent at $x$ to the geodesic curve joining $x$ to the destination of
the pedestrian at $x$. The convolution $\rho \conv \eta$ stands for
$\int_{\reali^2} \eta (x-\xi) \, \rho (t,\xi) \d\xi$, for a suitable
non-negative smooth kernel $\eta$. It represents the average density
measured, or felt, by the pedestrian at time $t$ in position $x$.

Below, we extend the results in~\cite{ColomboHertyMercier} proving the
global in time existence of the solutions
to~\eqref{eq:SCL}--\eqref{eq:Ped}. Moreover, we complete the stability
estimates with an estimate on the dependence of the solutions from $v$
and $\vec v$, see Theorem~\ref{thm:main1}. The resulting
model~\eqref{eq:SCL}--\eqref{eq:Ped} enjoys further remarkable
analytical properties. Indeed, a standard conservation law generates
a semigroup which is not differentiable with respect to the initial
data see~\cite[Section~1]{BressanGuerra} for an explicit example. On
the contrary, under suitable conditions, the semigroup generated
by~\eqref{eq:SCL}--\eqref{eq:Ped} turns out to be differentiable with
respect to the data. This allows to obtain, rigorously, necessary
conditions for optimality in various control problems based
on~\eqref{eq:SCL}--\eqref{eq:Ped}, see Section~\ref{sec:1} below
and~\cite[Section~4]{ColomboHertyMercier}.

\smallskip

Assuming that pedestrians adapt their path to the crowd density they
meet lead to the model presented
in~\cite{ColomboGaravelloMercier}. There, the speed law
\begin{equation}
  \label{eq:1}
  \direct
  =
  v (\rho)
  \left(
    \vec v (x)
    +
    \mathcal{I} (\rho)
  \right) \,.
\end{equation}
is considered. Again, $\vec v$ is the unit vector field describing the
preferred paths. But, contrary to~\eqref{eq:Ped}, here pedestrians may
deviate from it, due to the nonlocal term $\mathcal{I}$, which can be
assumed, for instance, of the form
\begin{equation}
  \label{eq:I}
  \mathcal{I} (\rho)
  =
  - \epsilon \,
  \frac{\nabla (\rho \conv \eta)}{\sqrt{1+\norma{\nabla (\rho\conv \eta)}^2}} \,.
\end{equation}
Again, $(\rho \conv \eta) (t,x)$ is the average density felt by the
pedestrian at $(t,x)$, so that~\eqref{eq:2} states that each
individual is ready to leave the preferred path in order to avoid
regions where the crowd density increases. The denominator
in~\eqref{eq:I} is a normalization factor, so that the modulus of
$\direct$ is essentially controlled by the function $v$
in~\eqref{eq:1}, a smooth non increasing function that vanishes at the
maximal density. We refer to~\cite{ColomboGaravelloMercier} for
further justifications of the choices leading
to~\eqref{eq:SCL}--\eqref{eq:1}. Below, we show through numerical
integrations that the formation of lanes, first noted
in~\cite{ColomboGaravelloMercier}, is present also in the present
multi-populations setting.

\smallskip

For both problems~\eqref{eq:SCL}--\eqref{eq:Ped}
and~\eqref{eq:SCL}--\eqref{eq:1}, we then consider the case of
several, say $n$, populations. By this, we mean that different groups
of pedestrians are considered, distinguished for instance by their
destination. This amounts to consider systems of the form
\begin{equation}
  \label{eq:SCLn}
  \partial_t \rho^i + \div (\rho^i \, \direct^i) = 0\,,
  \qquad i=1, \ldots, n
\end{equation}
where, in general, $\direct^i$ depends on the densities of all
populations: $\direct^i = \direct^i (\rho_1, \ldots, \rho_n)$. To
extend all the above well posedness and stability results, we rely
here essentially on~\cite{ColomboMercierRosini}, with the improvements
in~\cite{Lecureux}. We recall that systems of the form~\eqref{eq:SCLn}
are considered also in~\cite{CrippaMercier}, where measure theoretic
techniques are exploited. The interaction among different populations
is considered also in~\cite{DegondEtAl} through a macroscopic model
and in~\cite{CristianiPiccoliTosin} by means of a multiscale
model. For a general review about crowd dynamics models we refer
to~\cite{BellomoDogbe_review}.

\smallskip

The next section deals with the $n$-populations version
of~\eqref{eq:SCL}--\eqref{eq:Ped}. Then, Section~\ref{sec:2} is
devoted to the analogous extension of~\eqref{eq:SCL}--\eqref{eq:1},
presenting also a sample numerical integration. All proofs are
collected in Section~\ref{sec:TD}.

\smallskip

Throughout, we state and prove every result in $\reali^d$, for a
dimension $d \in \naturali$, $d > 0$, since the 2D case contains the
same difficulties as the general $d$-dimensional situation. By
$\reali^+$ we denote the interval $\left[0, +\infty \right[$.

\section{A Differentiable Model}
\label{sec:1}

The natural generalization of~\eqref{eq:Ped} to the case of $n$
populations is
\begin{equation}
  \label{eq:2}
  \direct^i
  =
  v^i( \rho^1 \conv \eta^1 + \ldots + \rho^n \conv \eta^n) \, {\vec{v}^i} (x)\,,
\end{equation}
so that $(\rho^i \conv \eta^i) (x)$ is an average of the values
attained by $\rho^i$ in $B(x,1)$. The map $v^i$ is the usual speed
law, typically required to be non increasing since at higher densities
the mean traffic speed is lower. Below, only the regularity of $v$ is
used. The vector $\vec{v}^i(x)$ is the direction of the pedestrian
belonging to the $i$-th population and situated at $x \in
\reali^2$. The presence of boundaries, obstacles or other geometric
constraints can be described through $\vec{v}^i$, see for
instance~\cite{ColomboFacchiMaterniniRosini}.

This section is devoted to the well posedness
of~\eqref{eq:SCLn}--\eqref{eq:2}, extending the results
in~\cite{ColomboHertyMercier} not only for what concerns the number of
populations, but also obtaining global in time existence and more
complete stability estimates. Before stating the main result, we
rigorously specify what we mean by \emph{solution}.

\begin{definition}
  \label{def:sol1}
  Let $T > 0$. Fix $\rho_o \in \L\infty (\reali^d; \reali^n)$. A weak
  entropy solution to~(\ref{eq:SCLn})--\eqref{eq:2} on $[0, T]$ is a
  bounded measurable map $\rho \in \C0 \left( [0, T]; \Lloc1(\reali^d;
    \reali^n) \right)$ whose $i$-th component $\rho^i$, for all $i=1,
  \ldots, n$, is a Kru\v zkov solution to the problem
  \begin{displaymath}
    \left\{
      \begin{array}{@{\,}l@{}}
        \partial_t \rho^i + \div \left( \rho^i \, \direct^i(t,x) \right) = 0
        \\
        \rho^i(0,x) = \rho_o^i(x)
      \end{array}
    \right.
    \quad \mbox{ where} \quad
    \direct^i(t,x) =
    v^i \!
    \left(
      \rho^1 (t) \conv \eta^1+ \ldots+
      \rho^n (t) \conv \eta^n
    \right)
    \, \vec{v}^i (x) \,.
  \end{displaymath}
\end{definition}

\noindent For the definition of Kru\v zkov solution,
see~\cite{Kruzkov} or~\cite[Paragraph~6.2]{DafermosBook}.  Above, the
convolution products in the arguments of $v^i$ are intended in the
sense
\begin{displaymath}
  \left(\rho^i (t) \conv \eta^i\right) (x)
  =
  \int_{\reali^d} \rho^i (t,\xi) \; \eta^i (x-\xi) \, \d\xi
  \qquad i=1, \ldots, n\,.
\end{displaymath}

The next Theorem summarizes various results
in~\cite{ColomboHertyMercier}, particularized
to~\eqref{eq:SCL}--\eqref{eq:Ped}.

\begin{theorem}
  \label{thm:main1}
  Fix $d,n \in \naturali$ with $d,n>0$.  Assume the following
  conditions:
  \begin{description}
  \item[($\boldsymbol{v}$)] $v^i \in (\C2 \cap \W2\infty)(\reali;
    \reali)$ for $i=1, \ldots, n$.
  \item[($\boldsymbol{\vec{v}}$)] ${\vec{v}^i} \in (\C2 \cap \W{2}{1})
    (\reali^d; \mathbb{S}^1)$ for $i=1, \ldots, n$.
  \item[($\boldsymbol{\eta}$)] $\eta^i \in (\C2 \cap \W2\infty)
    (\reali^d; [0,1])$ and $\norma{\eta^i}_{\L1} = 1$ for $i=1,
    \ldots, n$.
  \end{description}
  Then, there exists a semigroup $S \colon \reali^+ \times (\L1 \cap
  \L\infty \cap \BV) (\reali^d; \reali^n) \to (\L1 \cap \L\infty
  \cap \BV) (\reali^d; \reali^n)$ such that:
  \begin{enumerate}
  \item For all $\rho_o \in (\L1 \cap \L\infty \cap \BV) (\reali^d;
    \reali^n)$, for all $t\geq 0$, the orbit $t \mapsto S_t \rho_o$ is the
    unique solution to~\eqref{eq:SCL}--\eqref{eq:2} in the sense of
    Definition~\ref{def:sol1} with initial datum $\rho_o$.
    Furthermore, the map $t \mapsto S_t \rho_o$ is in $\C0 \left(
      \reali^+; \L1 (\reali^d; \reali^n) \right)$.
  \item For all $\rho_o \in (\L1 \cap \L\infty \cap \BV) (\reali^d;
    \reali^n)$, if $\rho_o^i \geq 0$ for $i = 1 \ldots, n$, then
    $(S_t\rho_o)_i \geq 0$ for all $t >0$.
  \item There exists a constant $\mathcal{L}$ such that for all
    $\rho_o \in (\L1 \cap \L\infty \cap \BV) (\reali^d; \reali^n)$,
    the corresponding solution satisfies for all $t \in \reali^+$
    \begin{displaymath}
      \tv\left(\rho(t) \right)
      \leq
      \left(
        \tv(\rho_o)
        +
        \mathcal{L} \, t \, \norma{\rho_o}_{\L\infty}
      \right) e^{\mathcal{L}  t}
      \quad \mbox{ and } \quad
      \norma{\rho(t)}_{\L\infty}
      \leq
      \norma{\rho_o}_{\L\infty} \, e^{\mathcal{L}  t} \,.
    \end{displaymath}
  \item Fix a positive $M$. Then there exists functions $L,
    \mathcal{A}_\eta, \mathcal{A}_v, \mathcal{A}_{\vec v} \in \C0
    (\reali^+; \reali^+)$ such for all $\rho_{o,1}, \rho_{o,2}$ in
    $\L1 (\reali^d; \reali^n)$ with
    $\max\left\{\norma{\rho_{o,i}}_{\L1} ,\,
      \norma{\rho_{o,i}}_{\L\infty} ,\, \tv (\rho_{o,i})\right\} \leq
    M$, for all $v_1, v_2$ satisfying $\boldsymbol{(v)}$, for all
    $\vec v_1, \vec v_2$ satisfying $\boldsymbol{(\vec v)}$ and for
    all $\eta_1, \eta_2$ satisfying $(\boldsymbol{\eta})$, the
    corresponding solutions $\rho_1, \rho_2$ satisfy, for all $t \in
    \reali^+$,
    \begin{eqnarray*}
      \norma{\rho_1(t) - \rho_2(t)}_{\L1}
      & \leq &
      \left(1+t L(t) \right) \,
      \norma{\rho_{o,1} - \rho_{o,2}}_{\L1}
      \\
      & &
      +
      t \, \mathcal{A}_\eta(t) \norma{\eta_1-\eta_2}_{\W1\infty}
      +
      t \, \mathcal{A}_{v}(t) \norma{v_1-v_2}_{\W1\infty}
      \\
      & &
      +
      t \, \mathcal{A}_{\vec v}(t)
      \left(
        \norma{\vec v_1-\vec v_2}_{\L\infty}
        +
        \norma{\vec v_1-\vec v_2}_{\W11}
      \right)\,.
    \end{eqnarray*}
  \item More regular initial data imply more regular solutions, in the
    sense that
    \begin{displaymath}
      \begin{array}{r@{\;}c@{\;}lcr@{\;}c@{\;}l@{\quad}r@{\;}c@{\;}l}
        \rho_o & \in & (\W11 \cap \L\infty)(\reali^d;\reali^n)
        & \Longrightarrow &
        \forall t & \in & \reali^+,
        & \rho(t) & \in & \W11(\reali^d;\reali^n)\, ,
        \\
        \rho_o & \in & \W1\infty (\reali^d;\reali^n)
        & \Longrightarrow &
        \forall t & \in & \reali^+,
        & \rho(t) & \in & \W1\infty(\reali^d;\reali^n) \,.
      \end{array}
    \end{displaymath}
    Furthermore, there exists a positive constant $C$ such that
    \begin{displaymath}
      \norma{\rho(t)}_{\W11}
      \leq
      (1+Ct) e^{Ct} \, \norma{\rho_o}_{\W11}
      \quad \mbox{ and } \quad
      \norma{\rho(t)}_{\W1\infty}
      \leq
      (1+Ct) e^{Ct} \, \norma{\rho_o}_{\W1\infty} \,.
    \end{displaymath}
  \item If $v$, $\vec v$ and $\eta$ are of class $\C3$, then
    \begin{displaymath}
      \begin{array}{r@{\;}c@{\;}lcr@{\;}c@{\;}l@{\quad}r@{\;}c@{\;}l}
        \rho_o & \in & (\W21 \cap \L\infty)(\reali^d; \reali^n)
        & \Longrightarrow &
        \forall t & \in & \reali^+,
        &
        \rho(t) & \in & \W21(\reali^d; \reali^n)\,,
        \\
        \rho_o & \in & \W2\infty(\reali^d; \reali^n)
        & \Longrightarrow &
        \forall t & \in & \reali^+,
        &
        \rho(t) & \in & \W2\infty(\reali^d; \reali^n)\,,
      \end{array}
    \end{displaymath}
    and for a suitable non--negative constant $C$, we have the
    estimates
    \begin{displaymath}
      \norma{\rho(t)}_{\W21}
      \leq
      e^{Ct}(1+Ct)^2 \, \norma{\rho_o}_{\W21} \,,
      \qquad
      \norma{\rho(t)}_{\W2\infty}
      \leq
      e^{Ct}(1+Ct)^2 \, \norma{\rho_o}_{\W2\infty} \,.
    \end{displaymath}
  \item If $v$ is of class $\C4$, for any initial data $\rho_o \in
    (\W2\infty \cap \W21 ) (\reali^d; \reali^n)$, $\sigma_o \in (\W11
    \cap \L\infty) (\reali^d; \reali^n)$ and for all time $t \in
    \reali^+$ the semigroup $S$ is strongly $\L1$ G\^{a}teaux
    differentiable in the direction $\sigma_o$. The derivative $\D
    S_t(\rho_o)(\sigma_o)$ of $S_t$ at $\rho_o$ in the direction
    $\sigma_o$ is
    \begin{displaymath}
      \D S_t(\rho_o)(\sigma_o) = \Sigma_t^{\rho_o}(\sigma_o) \, ,
    \end{displaymath}
    where $\Sigma^{\rho_o}$ is the linear semigroup generated by the
    Kru\v{z}kov solution to
    \begin{equation}
      \label{eq:linear}
      \left\{
        \begin{array}{l@{\qquad}r@{\,}c@{\,}l}
          \partial_t \sigma^i
          +
          \div \left(
            \sigma^i \direct^i(\rho)
            +
            \rho^i \,  \mathrm{D}\direct^i(\rho)(\sigma)
          \right)
          =
          0
          & (t,x) & \in & I \times \reali^d
          \\
          \sigma^i(0,x) = \sigma_o^i(x)
          & x & \in & \reali^d
        \end{array}
      \right.
    \end{equation}
    where $\rho (t) = S_t \rho_o$ for all $t\in \reali^+$ and
    $\direct^i$ is defined in~\eqref{eq:2}.
  \end{enumerate}
\end{theorem}

\noindent Note that the linear problem~\eqref{eq:linear} is the
equation that is obtained through a merely formal linearization
of~\eqref{eq:SCLn}--\eqref{eq:2}.

The above regularity results allow to state and prove the following
necessary condition for optimality.

\begin{proposition}
  \label{prop:optimal}
  Let $f\in \C{1,1}(\reali^n; \reali^+)$, $\psi \in \L\infty(\reali^+
  \times \reali^d; \reali^+)$ and assume that the
  problem~(\ref{eq:SCL})--(\ref{eq:2}) satisfies the assumptions
  at~6.~in Theorem~\ref{thm:main1}. Denote by $S \colon \reali^+
  \times (\L1 \cap \L\infty) (\reali^d; \reali^n) \to (\L1 \cap
  \L\infty) (\reali^d; \reali^n)$ the semigroup generated
  by~(\ref{eq:SCLn})--(\ref{eq:2}). Introduce the integral cost
  functional
  \begin{equation}
    \label{eq:J}
    J (\rho_{o} )
    =
    \int_{\reali^d}
    f \left( S_t \rho_o \right) \, \psi(t,x)
    \d{x} \,.
  \end{equation}
  Then, $J$ is strongly $\L\infty$ G\^ateaux differentiable in any
  direction $\sigma_o \in (\W11 \cap \L\infty) (\reali^d; \reali^n)$.

  Moreover, let $\Sigma \colon \reali^+ \times (\W11 \cap \L\infty)
  (\reali^d; \reali^n) \to (\W11 \cap \L\infty) (\reali^d; \reali^n)$
  be the linear semigroup generated by~(\ref{eq:linear}). Then,
  \begin{displaymath}
    DJ(\rho_o)(\sigma_o)
    =
    \int_{\reali^d}
    f'(S_t\rho_o) \, \Sigma_t^{\rho_o} (\sigma_o) \psi(t,x) \, \d{x} \, .
  \end{displaymath}

  If $\rho_\star \in (\L1 \cap \L\infty) (\reali^d; \reali^n)$ solves
  the problem
  \begin{displaymath}
    \textrm{find}\quad
    \rho_\star \in (\L1 \cap \L\infty) (\reali^d; \reali^n)
    \quad \mbox{ such that }\quad
    J (\rho_\star)
    =
    \min_{\rho_{o} \in (\L1 \cap \L\infty) (\reali^d; \reali^n)} J( \rho  )
  \end{displaymath}
  then, for all $\sigma_o \in (\L1 \cap \L\infty) (\reali^d;
  \reali^n)$,
  \begin{equation}
    \label{eq:last}
    \int_{\reali^d} f'(S_t \rho_\star) \, \Sigma_t^{\rho_\star} \sigma_o \, \psi(t,x)
    \, \d{x}
    =
    0 \,.
  \end{equation}
\end{proposition}

\noindent The proof directly follows from~\cite[Proposition~2.12 and
Theorem~4.2]{ColomboHertyMercier} and is hence omitted.

Remark that~\eqref{eq:SCLn}--\eqref{eq:2}provides an environment where
optimal control problems can be considered. On the other hand, no
uniform upper bound in $\L\infty$ is available.

\section{Pattern Formation}
\label{sec:2}

In the case of $n$ populations trying to avoid each other, we are led
to consider~\eqref{eq:SCLn} with
\begin{equation}
  \label{eq:Second}
  \direct^i
  =
  v^i (\rho^i) \,
  \left(
    {\vec{v}^i} (x)
    +
    \mathcal{I}^i (\rho^1, \ldots, \rho^n)
  \right)
  \quad \mbox{ for } \quad
  i=1, \ldots, n \,,
\end{equation}
where $\mathcal{I}^1, \ldots, \mathcal{I}^n$ are suitable nonlocal
functionals. According to~\eqref{eq:Second}, the velocity $\direct^i$
of the $i$-th population is the product of a scalar \emph{crowding
  factor} $v^i(\rho^i)$ with a vector ${\vec{v}^i}(x) +
\mathcal{I}^i(\rho^1, \ldots, \rho^n)$, which is the sum of a
\emph{preferred direction} ${\vec{v}^i} (x)$ and a \emph{deviation}
$\mathcal{I}^i (\rho^1, \ldots, \rho^n)$. The scalar $v^i(\rho^i)$
approximately gives the modulus of the speed. A further standard
condition on $v^i$ typically required in the engineering literature is
that $v^i$ be weakly decreasing. However, this assumption is here not
exploited.  The unit vector field ${\vec{v}^i}$ can be, for instance,
the vector tangent at $x$ to the geodesic that the individuals in the
$i$--th population would follow to get to their destination, if
unaffected by any other individual. The term $\mathcal{I}^i(\rho^1,
\ldots, \rho^n)$ describes how the $i$-th population deviates from its
preferred trajectory due to the interaction among individuals, both of
the same and of different populations. In general, it is a nonlocal
functional, since its value at any position $x$ depends on the
population densities averaged over a neighborhood of $x$. The present
setting generalizes the model in~\cite{ColomboGaravelloMercier}.

Below, we first address the main analytical properties
of~\eqref{eq:SCLn}--\eqref{eq:Second}, such as the existence of
solutions, their continuous dependence from the initial data and their
stability with respect to $\vec{v}^i$ and $\mathcal{I}^i$. Then, a
numerical integrations shows further qualitative properties of the
solutions to~\eqref{eq:SCLn}--\eqref{eq:Second}.

Denote by $R > 0$ a given maximal density.  Throughout, we also denote
the flux of the $i$-th population by $q^i(\rho^i) = \rho^i \,
v^i(\rho^i)$, for all $\rho^i \in [0, R]$.

Our starting point is the rigorous definition of solution
to~\eqref{eq:SCLn}--\eqref{eq:Second}, analogous to
Definition~\ref{def:sol1}.

\begin{definition}
  \label{def:sol2}
  Fix the initial datum $(\rho^1_{o}, \ldots, \rho^n_{o}) \in (\L1
  \cap \L\infty) (\reali^d; [0,R]^n)$.  A map $\rho \in \C0 \left(
    [0,T]; \L1(\reali^d; [0,R]^n) \right)$ is a \emph{weak entropy
    solution} to~\eqref{eq:SCLn}--\eqref{eq:Second} corresponding to
  the initial condition $(\rho^1_{o}, \ldots, \rho^n_{o})$ if, for $i
  = 1, \ldots, n$, $\rho^i$ is a Kru\v zkov solution to the Cauchy
  problem for the scalar conservation law
  \begin{displaymath}
    \left\{
      \begin{array}{l@{}}
        \partial_t \rho^i
        +
        \div \left( \rho^i\, v^i (\rho^i) \, \direct^i(t,x) \right) = 0
        \\
        \rho^i(0,x) = \rho_o^i (x)
      \end{array}
    \right.
    \quad \mbox{ where } \quad
    \direct^i(t,x)
    =
    {\vec{v}^i} (x) + \mathcal{I}^i \left( \rho_1(t), \ldots, \rho_n(t) \right) (x) \,.
  \end{displaymath}
\end{definition}

For the definition of Kru\v zkov solution we refer
to~\cite[Definition~1]{Kruzkov}.

\begin{theorem}
  \label{thm:main2}
  Assume the following conditions:
  \begin{description}
  \item[(v)] For $i=1, \ldots, n$, $v^i\in \C2([0,R]; \reali^+)$
    satisfies $v^i(R) = 0$.
  \item[($\boldsymbol{\vec v}$)] For $i = 1, \ldots, n$,
    ${\vec{v}^i}\in (\C2 \cap \W1\infty) (\reali^d; \reali^d)$ and
    $\div {\vec{v}^i} \in \W11 (\reali^d; \reali^{d\times d})$.
  \item[(I)] There exists a constant $C_I > 0$ such that the functional
    $\mathcal{I}^i \colon \L1(\reali^d; [0,R]^n) \to \C2 (\reali^d;
    \reali^d)$ satisfies, for $i = 1, \ldots, n$,
    \begin{displaymath}
      \begin{array}{@{}cl@{}}
        \forall \rho \in \L1(\reali^d; [0,R]^n)
        & \left\{
          \begin{array}{@{}rcl@{}}
            \norma{\nabla\mathcal{I}^i(\rho)}_{\L\infty(\reali^d; \reali^d)}
            & \leq &
            C_I \, \norma{\rho}_{\L1(\reali^d; \reali^n)} ,
            \\[5pt]
            \norma{\nabla \div \left(\mathcal{I}^i (\rho)\right)}_{\L1(\reali^d;
              \reali^{d\times d})}
            & \leq &
            C_I \, \norma{\rho}_{\L1(\reali^d; \reali^n)} .
          \end{array}
        \right.
        \\[15pt]
        \forall \rho, \rho' \in \L1(\reali^d; [0,R]^n)
        & \left\{
          \begin{array}{@{}rcl@{}}
            \norma{
              \mathcal{I}^i(\rho)
              -
              \mathcal{I}^i(\rho')}_{\L\infty(\reali^d; \reali^d)}
            & \leq &
            C_I \, \norma{\rho-\rho'}_{\L1(\reali^d; \reali^n)} ,
            \\[5pt]
            \norma{
              \div
              \left(
                \mathcal{I}^i(\rho)
                -
                \mathcal{I}^i(\rho') \right)}_{\L1(\reali^d;
              \reali)}
            & \leq &
            C_I \, \norma{\rho-\rho'}_{\L1(\reali^d; \reali^n)} .
          \end{array}
        \right.
      \end{array}
    \end{displaymath}
  \end{description}
  Then, there exists a semigroup $S \colon \reali^+ \times (\L1 \cap
  \BV) (\reali^d; [0,R]^n) \to (\L1 \cap \BV) (\reali^d; [0,R]^n)$
  such that
  \begin{enumerate}
  \item For all $\rho_o \in (\L1 \cap \BV) (\reali^d; [0,R]^n)$, the
    orbit $t \mapsto S_t \rho_o$ is the unique solution
    to~\eqref{eq:SCL}--\eqref{eq:Second} with initial datum $\rho_o$
    in the sense of Definition~\ref{def:sol2}.
  \item For all $\rho_o \in (\L1 \cap \BV) (\reali^d; [0,R]^n)$, the
    map $t \mapsto S_t \rho_o$ is in $\C0\left(\reali^+; \L1 (\reali^d;
      [0, R]^n) \right)$.
  \item For all $\rho_o \in (\L1 \cap \BV) (\reali^d; [0,R]^n)$, the
    following estimate holds:
    \begin{equation}
      \label{eq:6}
      \tv (S_t \rho_o)
      \leq
      \tv(\rho_o) \, e^{\kappa_o t}
      +
      d \, W_d \, e^{\kappa_o t} \norma{q}_{\L\infty([0,R])}
      (C_I + \norma{\div \vec v}_{\L\infty})
      \, t
      \,,
    \end{equation}
    where $W_d$ is defined in~\eqref{eq:Wd}.
  \item Fix $M>0$. Let $v_1, v_2$ satisfy~\textbf{(v)}, $\vec v_1,
    \vec v_2$ satisfy~\textbf{($\boldsymbol{v}$)} and $\mathcal{I}_1,
    \mathcal{I}_2$ satisfy~\textbf{(I)}. Then, there exist $b,c \in
    \C0 (\reali^+; \reali^+)$ such that for all $\rho_{o,1},
    \rho_{o,2} \in \L1 (\reali^d; [0,R]^d)$ with $\tv (\rho_{o,i})
    \leq M$ and for all $t \in \reali^+$
    \begin{eqnarray*}
      \!\!
      \norma{S_t\rho_{o,1} - S_t \rho_{o,2}}_{\L1}
      & \leq &
      \!\!\!
      (1+t\, e^{t\, b (t)}) \norma{\rho_{o,1} - \rho_{o,2}}_{\L1}
      \\
      & &
      \!\!\!
      +
      t\, c (t)
      \left(
        \norma{q_1 - q_2}_{\W{1}{\infty}}
        +
        \norma{\vec v_1 - \vec v_2}_{\L\infty}
        +
        \norma{\div (\vec v_1 - \vec v_2)}_{\L1}
      \right) .
    \end{eqnarray*}
    and the functions $b,c$ depend on $d$, $C_I$, $R$, 
    $\norma{q_1}_{\W1\infty}$, $\norma{\vec v_1}_{\W1\infty}$,
    $\norma{\div\vec v_1}_{\W{1}{1}}$, $\norma{\nabla\eta}_{\L\infty}$
    and on $M$.
  \end{enumerate}
\end{theorem}

\noindent Detailed expressions for the functions $b$ and $c$ are
available, together with the proof, in Paragraph~\ref{subsec:TD2}.

This theorem also provides a kind of maximum property since each
density remains bounded by $R$. However, it does not guarantee that
the sum $\rho_1(t) + \ldots + \rho_n(t)$ remains bounded by $R$.

\smallskip

Preliminary to any use of~\eqref{eq:SCLn}--\eqref{eq:Second} is the
choice of a specific $\mathcal{I}$. First, we consider the following
lemma that eases the construction of operators that
satisfy~\textbf{(I)}.

\begin{lemma}
  \label{lem:easy2}
  Let $\mathcal{N} \colon \reali^d \to \reali^d$ be defined by
  \begin{displaymath}
    \mathcal{N} (u) = \frac{u}{\sqrt{1+\norma{u}^2}} \,.
  \end{displaymath}
  If $\mathcal{I}$ satisfies~\textbf{(I)}, then so does $\mathcal{N}
  \circ \mathcal{I}$.
\end{lemma}

The proof consists in long but elementary computations and is hence
omitted. By Lemma~\ref{lem:easy2}, it is immediate to check
that~\eqref{eq:I} is satisfied by the operator $\mathcal{I}$ defined
in~\eqref{eq:I} in the case of one population, as soon as $\eta \in
\Cc2 (\reali^d; \reali^+)$ and $\norma{\eta}_{\L1} = 1$, considered
in~\cite{ColomboHertyMercier}. When more populations are present, it
is natural to consider different kinds of interactions. For instance,
the population $\rho^1$ might deviate from its preferred path
$\vec{v}^1$ due to the population $\rho^2$ pushing in the direction
$\vec{v}^2$, thus leading to consider the operator
\begin{equation}
  \label{eq:I12}
  \mathcal{I}^{12} (\rho)
  =
  \frac{
    \left(\rho^2 v(\rho^2 )\vec{v}^2 \right) \conv \eta
  }{
    \sqrt{1+\norma{\left(\rho^2  v(\rho^2)\vec{v}^2 \right) \conv \eta}^2}
  }\,,
\end{equation}
Under assumption~\textbf{(v)} on $v$ and with $\eta \in \Cc2
(\reali^d;\reali^+)$ such that $\norma{\eta}_{\L1}=1$, also
$\mathcal{I}^{12}$ as defined in~\eqref{eq:I12}
satisfies~\textbf{(I)}.

When these or other operators are to be considered together, then the
following lemma makes the verification of~\textbf{(I)} immediate.

\begin{lemma}
  \label{lem:easy}
  Let $\hat{\mathcal{I}}$ and $\check{\mathcal{I}}$
  satisfy~\textbf{(I)}. Fix any two functions $\alpha,\beta \in (\C2\cap\W2\infty)
  (\reali^d; \reali)$. Then, also $\alpha \, \hat{\mathcal{I}} + \beta
  \, \check{\mathcal{I}}$ satisfies~\textbf{(I)}.
\end{lemma}

\noindent The proof is elementary and hence omitted.

Differently from the model in Section~\ref{sec:1}, the semigroup
generated by~\eqref{eq:SCLn}--\eqref{eq:Second} is not proved to be
differentiable with respect to the initial data. On other hand, it
develops solutions with an apparently rich structure. Moreover, as
stated in Theorem~\ref{thm:main2}, an upper bound in $\L\infty$ is
available.

\subsection{Numerical Integration}

The study of the qualitative properties of the solutions
to~\eqref{eq:SCLn}--\eqref{eq:Second} is in general not amenable to
purely analytical tools. Numerical integrations, besides being of
interest from the point of view of the applications, show interesting
pattern formations. In the case~\eqref{eq:SCL}--\eqref{eq:1} of a
single population, the formation of queues was already noted
in~\cite{ColomboGaravelloMercier}.

The algorithm used is the Lax-Friedrichs method with dimensional
splitting. A uniform grid $(x_i,y_j)$ for $i = 1, \ldots, n_x$ and
$j=1, \ldots, n_y$ is introduced and the density $\rho$ is
approximated through the values $\rho_{ij}$ on this grid. At every
time step, the convolution in $\mathcal{I} (\rho)$ is computed through
products of the type $A_{ih} \rho_{hk} B_{kj}$, where the matrices $A$
and $B$ depend only on $\eta$.

In both the integrations, we use the same geometry and the same
numerical parameters. More precisely, the space available to the
pedestrians is the rectangle $\reali \times [-3, 3]$, while the
numerical domain is $[-8, 8] \times [-4, 4]$. Pedestrian may exit
along the segments $\{-8\} \times [-3, 3]$ and $\{8\} \times [-3,
3]$. The mesh size is $\Delta x = \Delta y = 0.025$.

The preferred path of each pedestrian is the sum $g+\delta$. The
vector $g$ is tangent to the geodesic towards the pedestrian's target
or $0$ when the pedestrians would not move. The vector $\delta$
describes the \emph{discomfort} of pedestrian when walking too near to
a wall. It is perpendicular to the walls and pointing towards the
interior of the room. Analytically, it also ensures the invariance of
the room, see~\cite{ColomboGaravelloMercier} for more details.  Its
maximum modulus is $\delta_{\max}$ along the walls at $\modulo{x_2} =
\Delta$ and decreases linearly towards the room interior, vanishing at
a distance $\delta_r$ from the walls. In the present integration we
set
\begin{equation}
  \label{eq:3}
  \delta_{\max} = 0.8
  \,,\qquad
  \delta_r = 0.75 \,.
\end{equation}

\subsection{Two Groups of People Crossing}

A classical situation considered in the engineering literature is that
of two groups of people moving in opposite directions and crossing
each other. Typically, \emph{lanes}, also called \emph{paths} or
\emph{trails} in the engineering literature, are formed, see for
instance~\cite{HelbingEtAlii2001, DaamenHoogendoorn2003} or~\cite{DegondEtAl} for a one dimensional description. They consist
of people going in the same direction. The
model~\eqref{eq:SCL}--\eqref{eq:2} captures this phenomenon. Indeed,
we consider~\eqref{eq:SCLn}--\eqref{eq:2} in the following setting:
\begin{equation}
  \label{eq:2D2p1}
  \left\{
    \begin{array}{l}
      \partial_t \rho^1
      +
      \div\left(
        \rho^1 \, v (\rho^1)
        \left(\vec{v}^1 (x)
          -
          \epsilon_{11} \,
          \frac{
            \nabla (\rho^1 \conv \eta)
          }{
            \sqrt{1+\norma{\nabla (\rho^1 \conv \eta)}^2}
          }
          -
          \epsilon_{12} \,
          \frac{
            \nabla (\rho^2 \conv \eta)
          }{
            \sqrt{1+\norma{\nabla (\rho^2 \conv \eta)}^2}
          }
        \right)
      \right)
      =
      0
      \\
      \partial_t \rho^2
      +
      \div\left(
        \rho^2 \, v (\rho^2)
        \left(
          \vec v^2(x)          -
          \epsilon_{21} \,
          \frac{
            \nabla (\rho^1 \conv \eta)
          }{
            \sqrt{1+\norma{\nabla (\rho^1 \conv \eta)}^2}
          }
          -
          \epsilon_{22} \,
          \frac{
            \nabla (\rho^2 \conv \eta)
          }{
            \sqrt{1+\norma{\nabla (\rho^2 \conv \eta)}^2}
          }
        \right)
      \right)
      =
      0
    \end{array}
  \right.
\end{equation}
where the $\rho^1$ population moves to the right and the $\rho^2$
population move to the left. Above, we assume that each pedestrians
wants to avoid entering regions with increasing densities and that the
repulsion towards pedestrians of the other population is bigger than
that due to the proper population. For the sake of simplicity, we use
the same kernel $\eta$ and the same speed $v$ in both equations. We
consider the specific situation defined by
\begin{equation}
  \label{eq:lanes1}
  \begin{array}{@{}rcl@{\quad}rcl@{}}
    \vec{v}^1
    & = &
    \left[
      \begin{array}{@{}c@{}}
        1\\0
      \end{array}
    \right]
    + \delta\,,
    &
    \eta (x_1,x_2)
    & = &
    \left[1-(2\,x_1)^2\right]^3
    \left[1-(2\,x_2)^2\right]^3
    \caratt{[-0.5, 0.5]^2} (x_1,x_2)\,,
    \!\!\!\!\!\!
    \\[15pt]
    \vec{v}^2
    & = &
    \left[
      \begin{array}{@{}c@{}}
        -1\\0
      \end{array}
    \right]
    + \delta\,,
    &
    v (\rho)
    & = &
    4\, (1-\rho)\,,
    \qquad
    \begin{array}{rcl@{\quad}rcl}
      \epsilon_{11} & = & 0.3\,, & \epsilon_{12} & = & 0.7\,,
      \\
      \epsilon_{21} & = & 0.7\,, & \epsilon_{22} & = & 0.3\,.
    \end{array}
  \end{array}
  \!
\end{equation}
As initial datum we choose
\begin{equation}
  \label{eq:4}
  \begin{array}{rcl}
    \rho_o^1 (x_1,x_2)
    & = &
    0.9 \cdot \caratt{[-6.4, -3.2]\times [-2.4,2.4]} (x_1, x_2) \,,
    \\
    \rho_o^2 (x_1,x_2)
    & = &
    0.7 \cdot \caratt{[3.2, 6.4]\times [-2.4,2.4]} (x_1, x_2) \,.
  \end{array}
\end{equation}
The resulting numerical integration in Figure~\ref{fig:NI1} shows
several lanes forming. First, when most of the populations are still
separated, the lanes in the two populations are rather symmetric, in
spite of the different initial densities
\begin{figure}[htpb]
  \centering%
  \includegraphics[width=0.25\textwidth, trim=120 40 121 30]{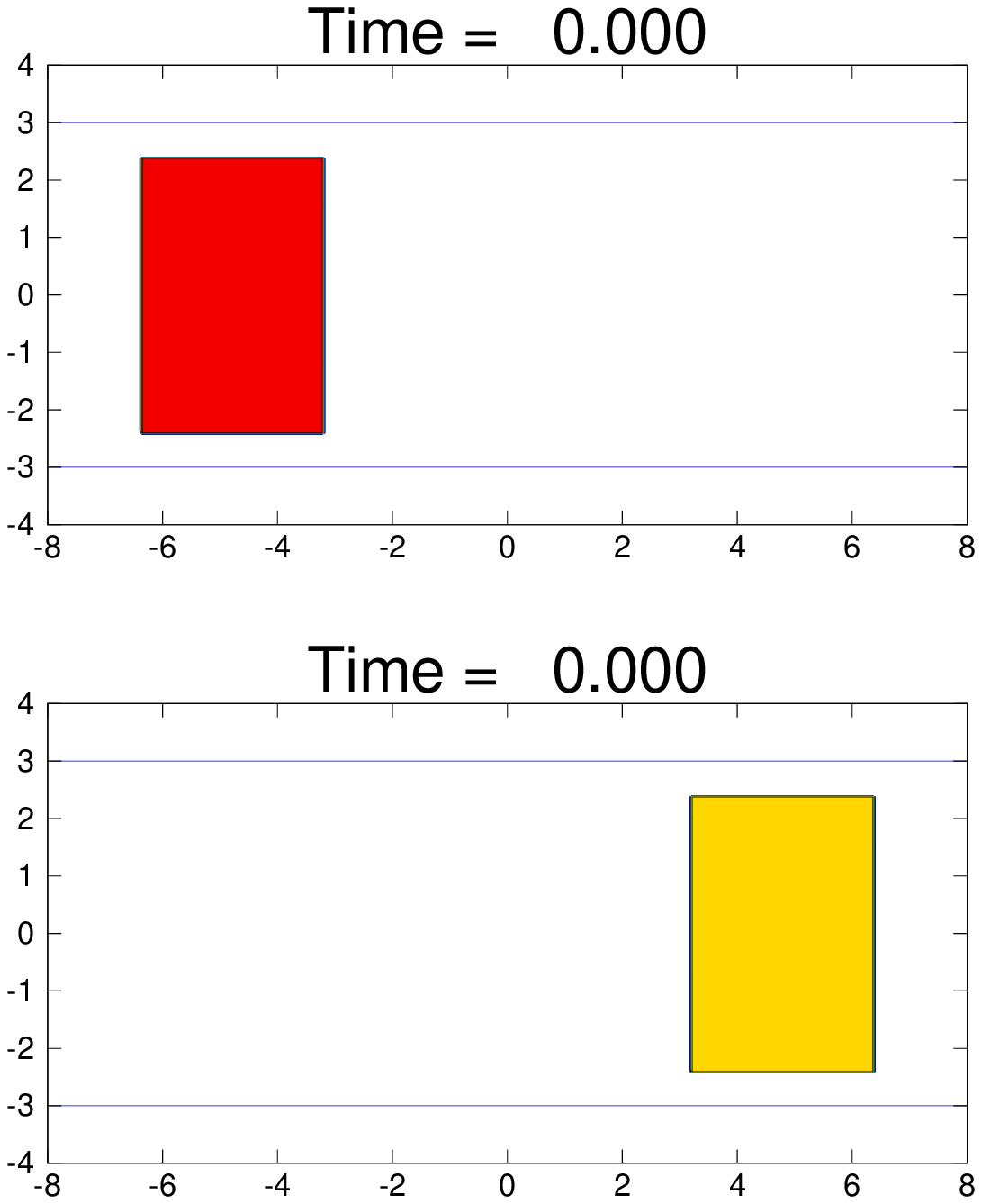}%
  \includegraphics[width=0.25\textwidth, trim=120 40 121 30]{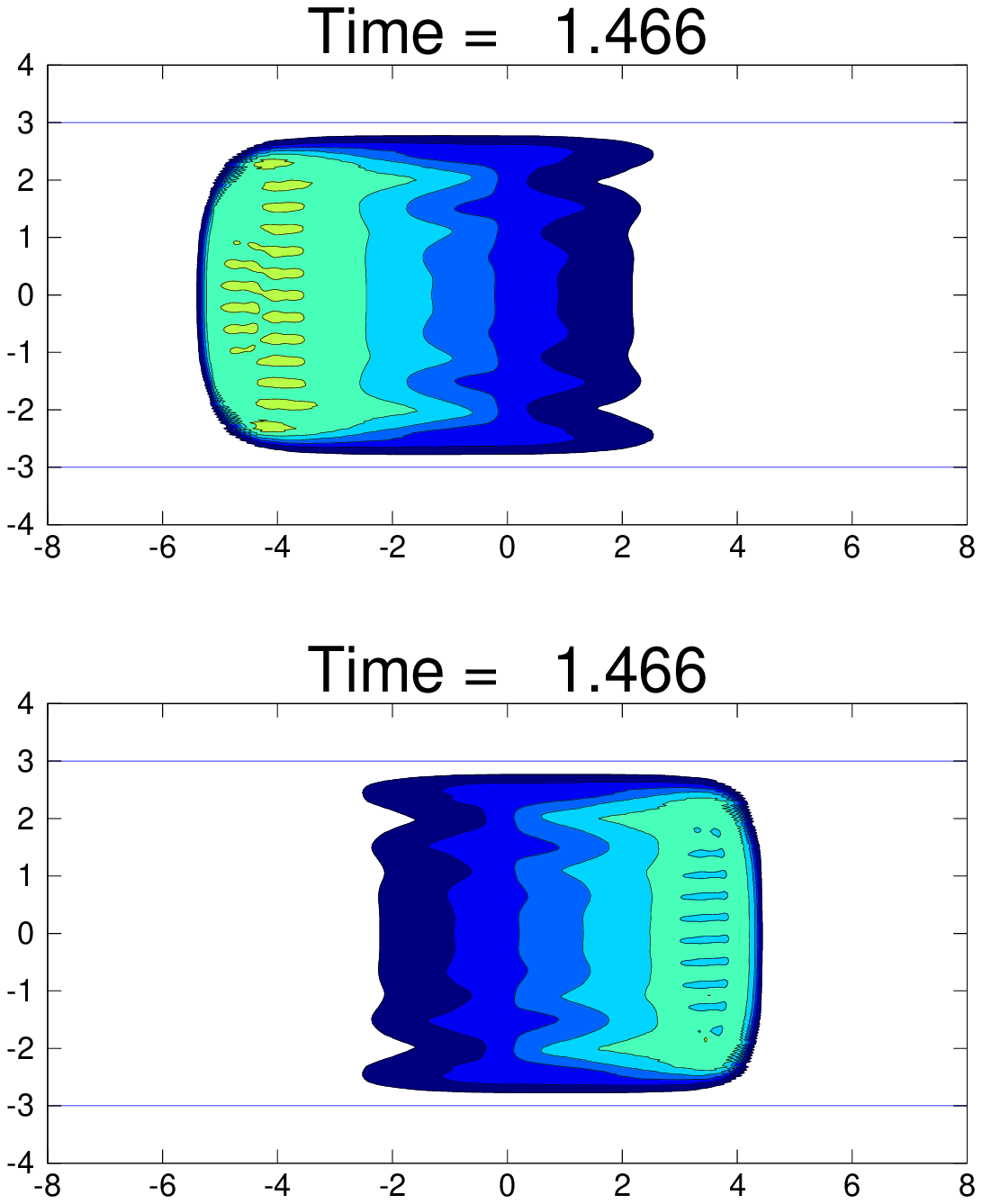}%
  \includegraphics[width=0.25\textwidth, trim=120 40 121 30]{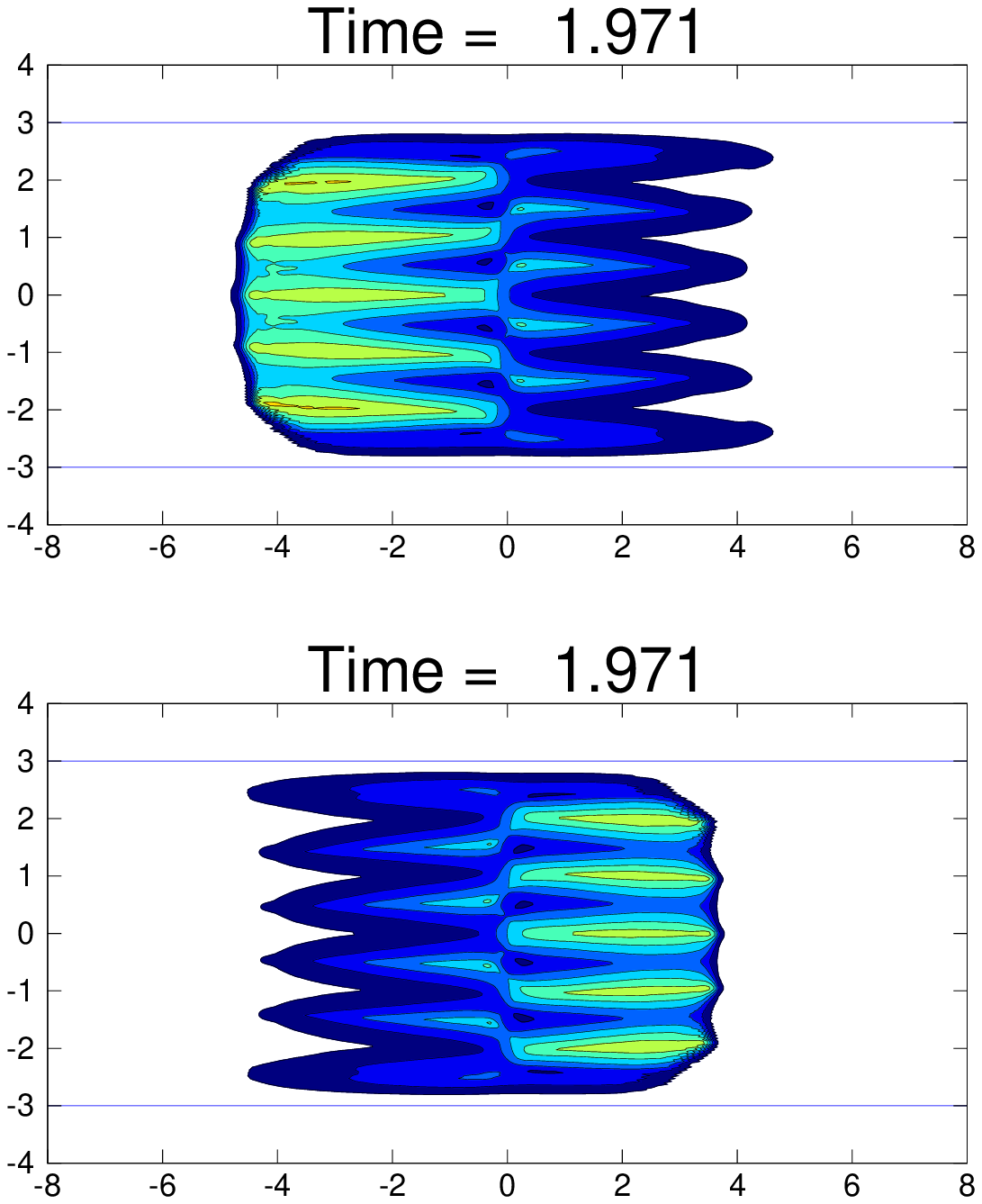}%
  \includegraphics[width=0.25\textwidth, trim=120 40 121 30]{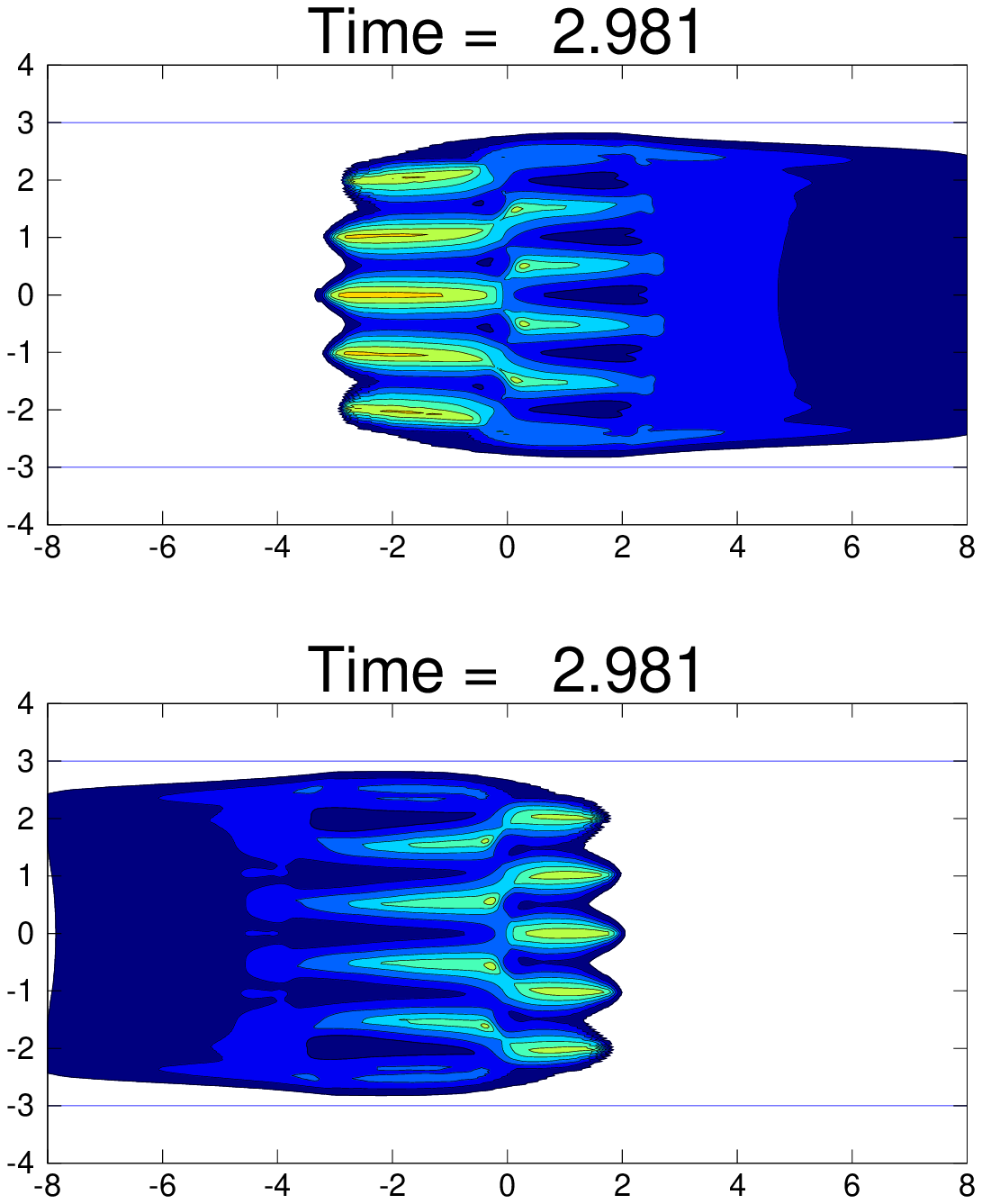}%
  \caption{Numerical integration of~\eqref{eq:2D2p1} with
    data~\eqref{eq:4} and parameters as
    in~\eqref{eq:3}--\eqref{eq:lanes1}. Above, the population $\rho^1$
    moving to the right and, below, $\rho^2$ moving to the left. Note
    the lanes that are formed. First, due to the self-interaction
    within each populations and then due to the crossing between the
    two populations. The latter lanes of different populations do not
    superimpose.}
  \label{fig:NI1}
\end{figure}
Then, when the interaction between $\rho^1$ and $\rho^2$ gets
relevant, lanes sharply bend their trajectory so that they are not
superimposed. In this way, pedestrians lower how much they block each
other. After the interaction, the density is so low that no more
patterns arise.

\subsection{Part of an Audience Leaving a Room}

A sample situations developing various interesting features is the
following.  Two populations are initially uniformly distributed in the
same region. At time $t=0$, the first populations starts moving
towards an exit, on the right in Figure~\ref{fig:NI}), while the
second moves only to let the first one pass. We describe this
situation through~\eqref{eq:SCLn}--\eqref{eq:Second}. The preferred
path of $\rho^1$ is tangent to $\vec{v}^1 = g + \delta$, where $g$ is
the geodesic vector field directed toward the right exit and $\delta$
is the discomfort, as above. The $\rho^1$ population then deviates
from this trajectory only to avoid the other population. On the
contrary, the preferred path of $\rho^2$ is tangent to $\vec{v}^2 =
\delta$, since this population moves only to avoid the walls and the
$\rho^1$ population. All this leads to the system
\begin{equation}
  \label{eq:2D2p}
  \left\{
    \begin{array}{l}
      \partial_t \rho^1
      +
      \div\left(
        \rho^1 \, v (\rho^1)
        \left(\vec{v}^1 (x)
          - \epsilon \,
          \frac{
            \nabla (\rho^2 \conv \eta)
          }{
            \sqrt{1+\norma{\nabla (\rho^2 \conv \eta)}^2}
          }
        \right)
      \right)
      =
      0 \,,
      \\
      \partial_t \rho^2
      +
      \div\left(
        \rho^2 \, v (\rho^2)
        \left(
          \vec v^2(x) - \epsilon \,
          \frac{
            \nabla (\rho^1 \conv \eta)
          }{
            \sqrt{1+\norma{\nabla (\rho^1 \conv \eta)}^2}
          }
        \right)
      \right)
      =
      0 \,.
    \end{array}
  \right.
\end{equation}
We consider the specific situation defined by
\begin{equation}
  \label{eq:lanes}
  \begin{array}{@{}rcl@{\quad}rcl@{}@{\quad}rcl@{}}
    \vec{v}^1
    & = &
    \left[
      \begin{array}{@{}c@{}}
        1\\0
      \end{array}
    \right]
    + \delta\,,
    &
    \eta (x_1, x_2)
    & = &
    \left[1-(2\,x_1)^2\right]^3
    \left[1-(2\,x_2)^2\right]^3
    \caratt{[-0.5, 0.5]^2} (x_1,x_2)\,,
    \!\!\!\!\!\!
    \\[15pt]
    \vec{v}^2
    & = &
    \delta\,,
    &
    v (\rho)
    & = &
    4\, (1-\rho)\,,
    \qquad\qquad
    \epsilon
    \; = \;
    0.3\,.
    \!\!\!\!\!\!\!\!\!\!\!\!
  \end{array}
  \!
\end{equation}
with initial data
\begin{equation}
  \label{eq:5}
  \rho_o^i
  =
0.5 \; \caratt{[-6.4, \, -3.2] \times [-2.4, \, 2.4]}\,,
  \qquad i = 1,2 \,,
\end{equation}
The resulting numerical integration is in Figure~\ref{fig:NI}.
\begin{figure}[htpb]
  \centering%
  \includegraphics[width=0.25\textwidth, trim=120 40 121
  30]{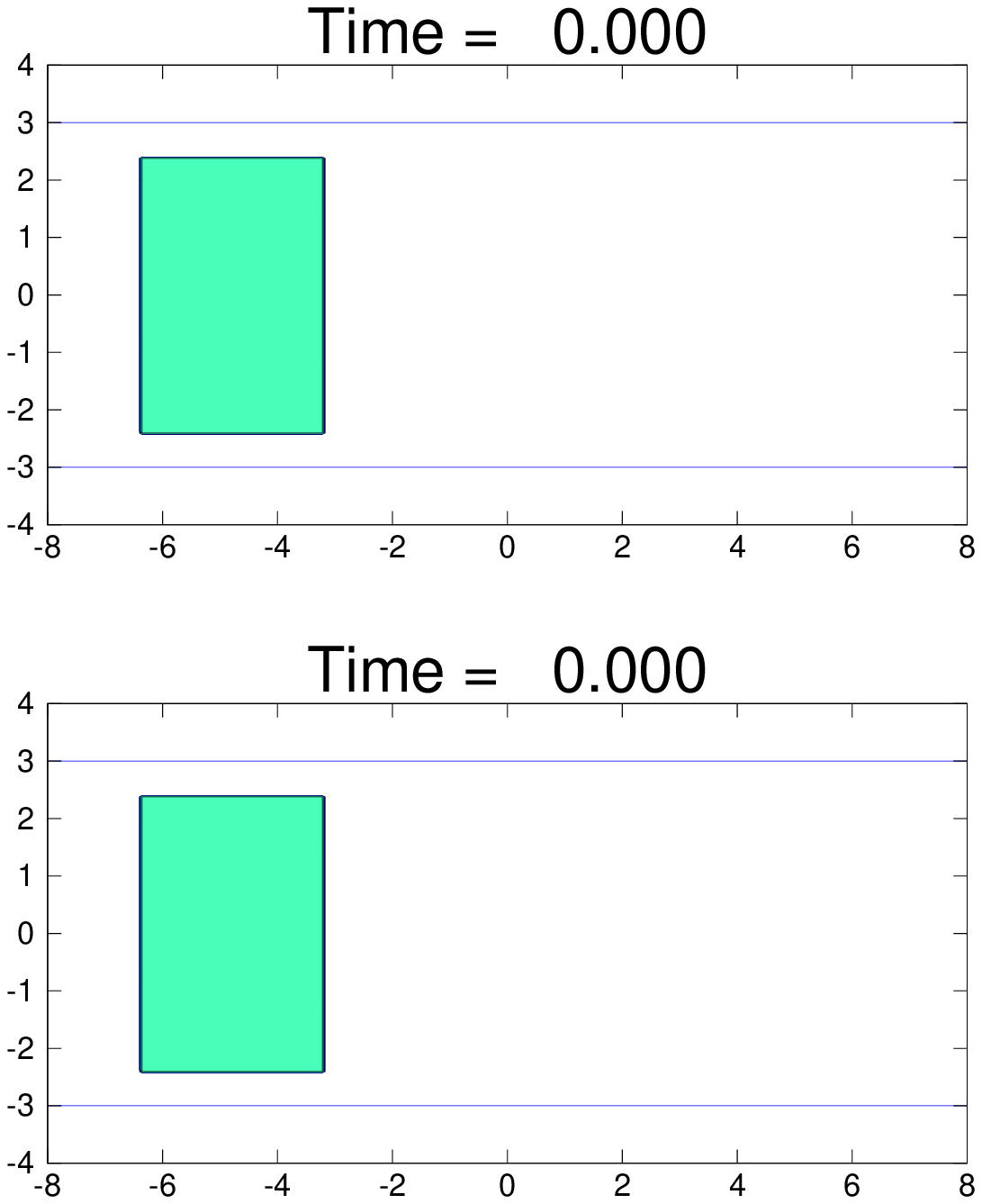}%
  \includegraphics[width=0.25\textwidth, trim=120 40 121
  30]{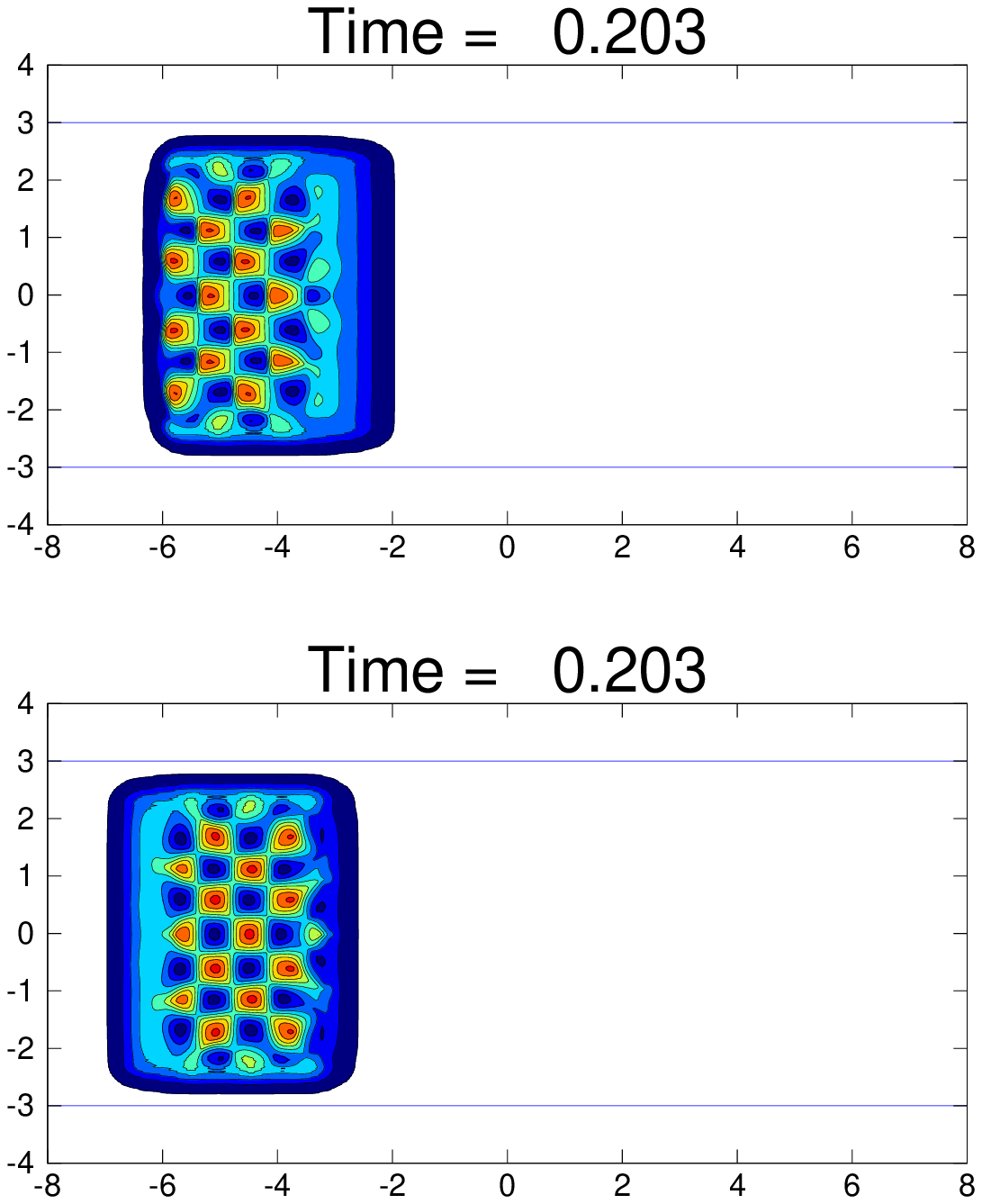}%
  \includegraphics[width=0.25\textwidth, trim=120 40 121
  30]{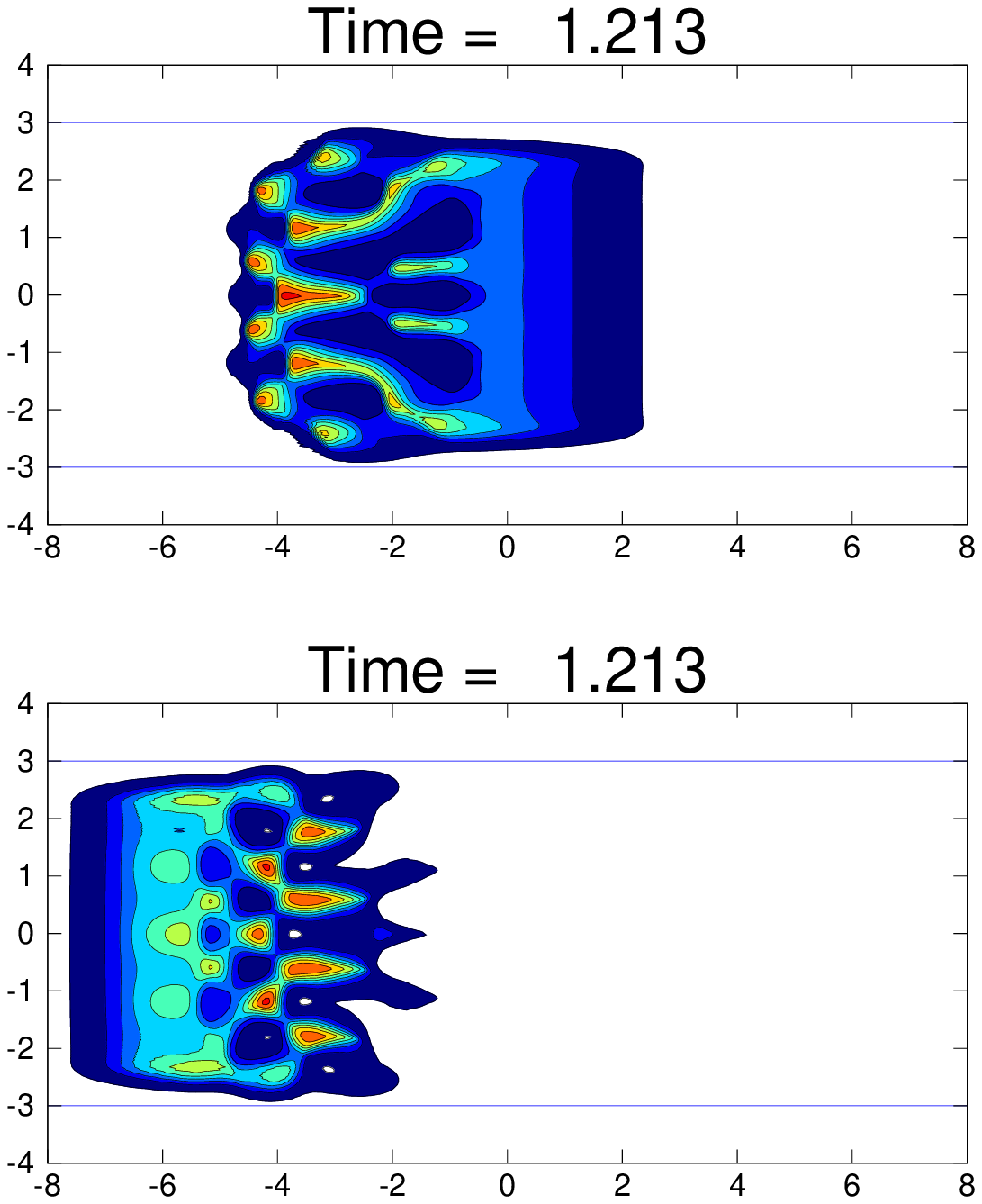}%
  \includegraphics[width=0.25\textwidth, trim=120 40 121
  30]{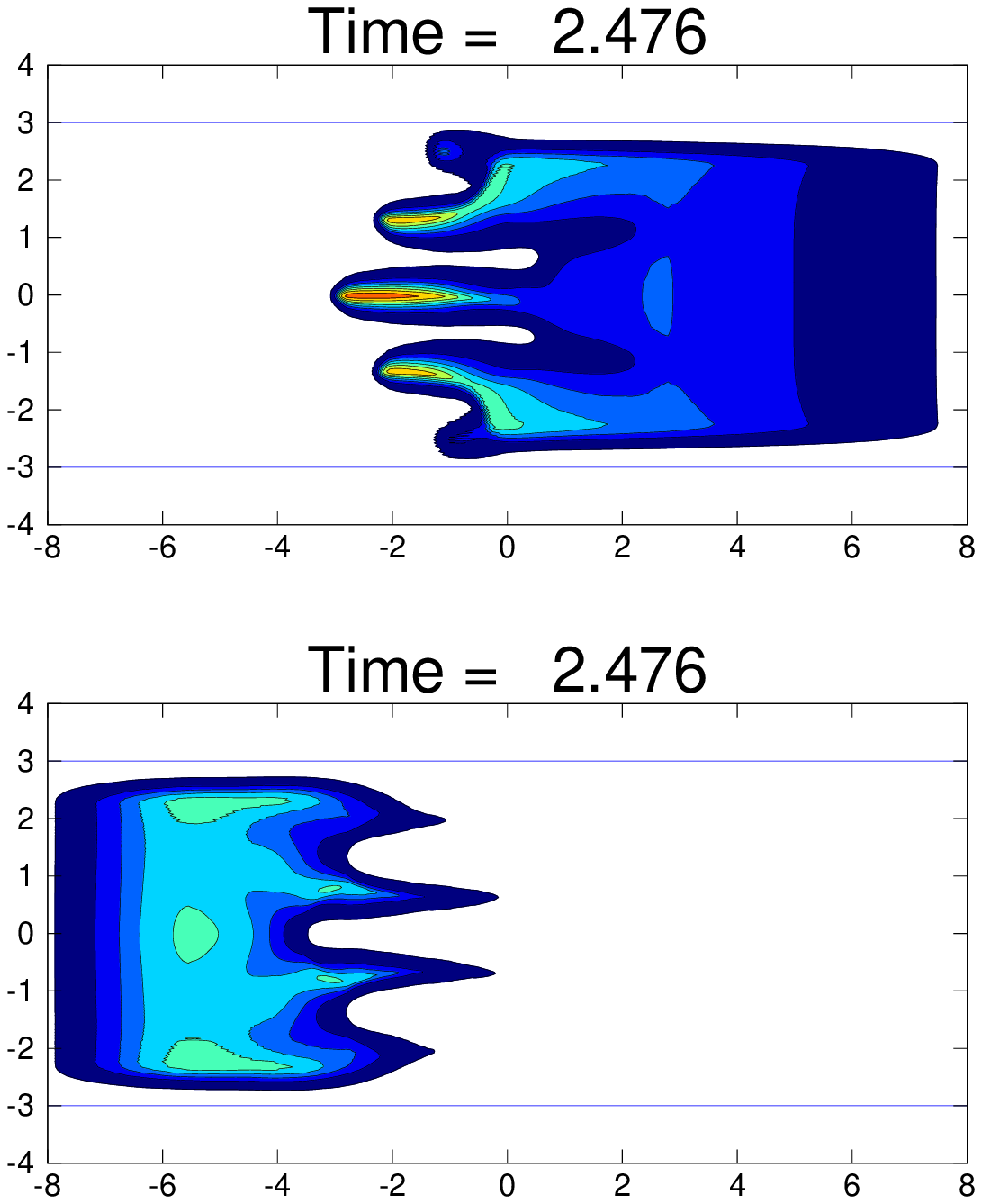}%
  \caption{Numerical integration of~\eqref{eq:2D2p} with data and
    parameters as in~\eqref{eq:3}--\eqref{eq:lanes}. Above, the
    population $\rho^1$ and, below, $\rho^2$. Note first the formation
    of small clusters separating the two populations, then lanes and,
    finally, a sort of fingering.}
  \label{fig:NI}
\end{figure}
Note that the initial distributions of both populations are
uniform. Very quickly, patterns start to form. First,
\emph{clustering}: the different populations separate forming
separated peaks of high density.  Then, population $\rho^1$ starts
moving significantly to the right along \emph{lanes} that follow paths
in regions of low $\rho^2$ density. Finally, the two populations are
almost separated, with a sort of \emph{fingering} remaining in the
region when they are superimposed. Concerning the initial clustering,
we underline the analogy with~\cite{CristianiPiccoliTosin}, where a
population is described through a measure possibly containing also an
atomic part. Concerning the latter fingering, we borrowed here this
term from the entirely different case of, for instance, thin films,
see~\cite{LevyShearer2006}).

\section{Technical Details}
\label{sec:TD}

Below, for any $\rho \in \L1(\reali^d; \reali^n)$,
$\norma{\rho}_{\L1(\reali^d; \reali^n)}$ stands for $\sum_{i=1}^n
\norma{\rho^i}_{\L1(\reali^d; \reali)}$. Similarly, $\tv (\rho) =
\sum_{i=1}^n \tv (\rho^i)$. More generally, if not otherwise stated,
the norm of a vector $v$ in $\reali^n$ is the $1$-norm,
i.e.~$\norma{v} = \sum_{i=1}^n \modulo{v^i}$.
The following constant enters several estimates below:
\begin{equation}
  \label{eq:Wd}
  W_d
  =
  \int_0^{\pi/2} (\cos\theta)^d\d{\theta}\,.
\end{equation}

The analytical tools below are based on the classical Kru\v zkov
work~\cite{Kruzkov}, see also~\cite[Chapter~VI]{DafermosBook}. The
stability estimates are taken from~\cite{ColomboMercierRosini,
  Lecureux}.

First, we study the following Cauchy problem for a scalar non-linear
conservation law:
\begin{equation}
  \label{eq:scalar}
  \left\{
    \begin{array}{l}
      \partial_t \rho
      +
      \div \left(
        q(   \rho )\, \direct (t,x)\right)=0
      \\
      \rho(0) = \rho_o\,.
    \end{array}
  \right.
\end{equation}

\begin{lemma}
  \label{lem:scalar}
  Assume that
  \begin{equation}
    \label{eq:Lemma}
    \begin{array}{rcl}
      q & \in & \C2 (\reali^+; \reali^+)
      \mbox{ satisfies } q(0)=0\,,
      \\
      \direct
      & \in &
      \C0 (\reali^+\times\reali^d;  \reali^d)
      \mbox{ satisfies }
      \left\{
        \begin{array}{l@{}}
          \div  \direct (t) \in \W11(\reali^d; \reali^d)
          \\
          \direct (t) \in \W1\infty (\reali^d; \reali^d)
        \end{array}
        \right.
        \mbox{ for all }t \in \reali^+ \,,
        \\
        \rho_o & \in & (\L1\cap \L\infty) (\reali^d; \reali^+) \,.
      \end{array}
  \end{equation}
  Then, there exists a unique weak entropy solution $\rho\in
  \C0\left(\reali^+,\L1(\reali^d;\reali^+)\right)$
  to~(\ref{eq:scalar}).

  If furthermore $\rho_o \in \BV(\reali^d; \reali)$, then $\rho(t) \in
  \BV (\reali^d; \reali)$ for all $t\in \reali^+$ and, denoting
  $R_T = \norma{\rho(t)}_{\L\infty([0,T]\times\reali^d)}$,
  \begin{equation}
    \label{eq:bv}
    \tv \left(\rho(t)\right)
    \leq
    \tv(\rho_o) e^{\kappa_o t}
    +
    C_d \, e^{\kappa_o t} \norma{q}_{\L\infty([0,R_T])}\int_0^t \int_{\reali^d}
    \norma{\nabla\div \direct(\tau,x)}
    \, \d{x} \, \d{\tau}\,,
  \end{equation}
  where $\kappa_o= (2d +1)\norma{q'}_{\L\infty([0,R_T])} \norma{\nabla
    \direct}_{\L\infty([0,T]\times\reali^d)}$ and  $C_d = d\, W_d$, with $W_d$ defined in~\eqref{eq:bv}.

  Let now $v_1,v_2$, $\direct_1, \direct_2$ and $\rho_{o,1},
  \rho_{o,2}$ satisfy~\eqref{eq:Lemma}. Call $\rho_1, \rho_2$ the
  solutions to
  \begin{equation}
    \label{eq:stab}
    \left\{
      \begin{array}{l@{}}
        \partial_t \rho_1
        +
        \div \left (\rho_1 \, v_1(\rho_1) \, \direct_1(t,x) \right)
        =
        0
        \\
        \rho_1(0) = \rho_{o,1}
      \end{array}
    \right.
    \quad \mbox{and} \quad
    \left\{
      \begin{array}{l@{}}
        \partial_t \rho_2
        +
        \div \left (\rho_2 \, v_2(\rho_2) \, \direct_2(t,x) \right)
        =
        0
        \\
        \rho_2(0) = \rho_{o,2}
      \end{array}
    \right.
  \end{equation}
  Then, renaming
  $R_T=\max\{\norma{\rho_1}_{\L\infty([0,T]\times\reali^d)},
  \norma{\rho_2}_{\L\infty([0,T]\times\reali^d)}\}$,
  \begin{eqnarray*}
    \norma{\rho_1(t) - \rho_2(t)}_{\L1}
    & \leq &
    \norma{\rho_{o,1} - \rho_{o,2}}_{\L1}
    \\
    & &
    +
    t \, e^{\kappa_o t}
    \left(
      \norma{q_2'}_{\L\infty([0,R_T])}
      \norma{\direct_1 - \direct_2}_{\L\infty}
      +
      \norma{q_1'- q_2'}_{\L\infty([0,R_T])}
      \norma{\direct_1}_{\L\infty}
    \right)
    \\
    & &
    \times
    \left[
      \tv(\rho_{o,1})
      +
      C_d \, \norma{q_1}_{\L\infty([0,R_T])}
      \int_0^t \int_{\reali^d}
      \norma{\nabla\div \, \direct_1(\tau, x)}
      \d{x} \d{\tau}
    \right]
    \\
    & &
    +
    \norma{q_1}_{\L\infty([0,R_T])}
    \int_0^t \int_{\reali^d}
    \modulo{\div\left(\direct_1(\tau, x)-\direct_2(\tau, x)\right)}
    \d{x} \d{\tau}
    \\
    & &
    +
    \norma{q_1 - q_2}_{\L\infty([0,R_T])}
    \int_0^t \int_{\reali^d}
    \modulo{\div \direct_2(\tau, x)} \d{x} \d{\tau} \,.
  \end{eqnarray*}
\end{lemma}

\begin{proof}
  We consider the following steps separately.

  \paragraph{Existence.}
  Let $f(t,x,\rho) = q(\rho) \, \direct(t,x)$. Then, we have $f\in
  \C0(\reali^+\times\reali^d\times \reali; \reali^d)$ and $f(t, \,
  \cdot\, , \, \cdot \, ) \in \C2(\reali^d\times \reali; \reali^d)$
  for any $t \in \reali^+$. Through elementary computations the
  following implications can be checked:
  \begin{displaymath}
    \begin{array}{rcl@{\quad\Rightarrow\quad}rcl}
      \partial_\rho f (t,x,\rho)
      & = &
      q'(\rho) \, \direct(t,x)
      &
      \partial_\rho f
      & \mbox{is}& \mbox{bounded on }[0,T] \times \reali^d \times [-A,A]
      \\
      \div f (t,x,\rho)
      & = &
      q(\rho)\, \div \direct (t, x)
      &
      \div f
      & \in &
      \L\infty([0,T] \times \reali^d\times [-A,A], \reali)
      \\
      \partial_\rho \div f (t, x, \rho)
      & = &
      q'(\rho) \,\div \direct (t, x)
      &
      \partial_\rho \div f
      & \in &
      \L\infty([0,T]\times \reali^d\times [-A,A])
    \end{array}
  \end{displaymath}
  for all $A \in \reali^+$. Hence, \cite[Theorem~4]{Kruzkov} can be
  applied and the existence of solutions follows.

  \paragraph{Maximum principle.}
  Since $q(0)=0$, $\rho\equiv 0$ is solution to~(\ref{eq:scalar}). The
  maximum principle of Kru\v zkov~\cite[Theorem~3]{Kruzkov} then
  ensures that $\rho_o\geq 0$ implies $\rho(t)\geq 0$ for all $t\geq
  0$.

  \paragraph{$\BV$ bound.}
  The $\BV$ bound follows from~\cite[Theorem~2.2]{Lecureux}, which can
  be applied since $\nabla\partial_\rho f = q'(\rho) \, \nabla \direct
  (t,x)$, so that $\nabla\partial_\rho f \in
  \L\infty([0,T]\times\reali^d\times[-A, A])$. Moreover, for any $A\geq 0$,
  \begin{displaymath}
    \int_0^T \!\!\! \int_{\reali^d} \!
    \norma{\nabla \div f(t,x, \, \cdot \,)}_{\L\infty([-A,A])} \d{x} \d{t}
    \!=\!
    \norma{q(\rho)}_{\L\infty([-A,A])}
    \int_0^T \!\!\! \int_{\reali^d} \!
    \norma{\nabla\div  \vec v (t,x)} \d{x} \d{t}
    {<}
    {+}\infty.
  \end{displaymath}
  Let $\kappa_o = (2 d+1)\norma{q'}_{\L\infty([0,R_T])} \norma{\nabla
    \direct}_{\L\infty([0,T]\times\reali^d)}$. Then, for any $t\geq
  0$, we have:
  \begin{displaymath}
    \tv \left( \rho(t) \right)
    \leq
    \tv(\rho_o) e^{\kappa_o t}
    +
    C_d \, e^{\kappa_o t} \norma{q}_{\L\infty([0,R_T])}
    \int_0^t \int_{\reali^d}
    \norma{\nabla\div \, \direct(\tau,x)} \d{x} \d{\tau} \,,
  \end{displaymath}
  with $C_d = d \, W_d$, and $W_d$ as defined in~\eqref{eq:Wd}.

  \paragraph{Stability estimate.}
  The stability estimate follows from~\cite[Theorem~2.6]{Lecureux}. We
  have, for any $A\geq 0$,
  \begin{eqnarray*}
    & &
    \int_0^T \int_{\reali^d}
    \norma{\div (f_1 - f_2)(t,x, \, \cdot \,)}_{\L\infty([-A,A])} \d{x} \d{t}
    \\
    & = &
    \norma{q_1 (\rho_1)}_{\L\infty([-A,A])}
    \int_0^T \int_{\reali^d}
    \modulo{\div \direct_1(\tau,x) - \div\direct_2(\tau, x)}
    \d{x} \d{t}
    \\
    & &
    +
    \norma{q_1 - q_2 }_{\L\infty([-A,A])}
    \int_0^T \int_{\reali^d}
    \modulo{\div\direct_2(\tau, x)} \d{x} \d{t}
    <
    +\infty \,.
  \end{eqnarray*}
  Then, with
  $R_T=\max\{\norma{\rho_1}_{\L\infty([0,T]\times\reali^d)},
  \norma{\rho_2}_{\L\infty([0,T]\times\reali^d)}\}$, we get
  \begin{eqnarray*}
    \norma{\rho_1(t) - \rho_2(t)}_{\L1}
    & \leq &
    \norma{\rho_{o,1} - \rho_{o,2}}_{\L1}
    \\
    & &
    +
    t \, e^{\kappa_o t}
    \left(
      \norma{q_2'}_{\L\infty([0,R_T])}
      \norma{\direct_1 - \direct_2}_{\L\infty}
      +
      \norma{q_1' - q_2'}_{\L\infty([0,R_T])} \norma{ \direct_1 }_{\L\infty}
    \right)
    \\
    &&
    \quad  \times
    \left[
      \tv(\rho_{o,1})
      +
      C_d \, \norma{q_1}_{\L\infty([0,R_T])}
      \int_0^t \int_{\reali^d}
      \norma{\nabla\div \, \direct_1(\tau, x)} \d{x} \d{\tau}
    \right]
    \\
    & &
    + \norma{q_2}_{\L\infty([0,R_T])}
    \int_0^t\int_{\reali^d}
    \modulo{\div \left(\direct_1(\tau, x) - \direct_2(\tau, x) \right)}
    \d{x} \d{\tau}
    \\
    & &
    +
    \norma{q_1 - q_2}_{\L\infty([0,R_T])}
    \int_0^t \int_{\reali^d}
    \modulo{\div\direct_1(\tau, x)} \d{x} \d{\tau} \,,
  \end{eqnarray*}
  completing the proof.
\end{proof}

\begin{corollary}\label{cor:scalar}
  In the same setting as Lemma~\ref{lem:scalar}, if
  $q(\rho)=q_1(\rho)=q_2(\rho)=\rho$ then $q'(\rho)=1$, $q''(\rho)=0$
  and the stability estimate reduces to
  \begin{eqnarray*}
    \norma{\rho_1(t) - \rho_2(t)}_{\L1}
    & \leq &
    \norma{\rho_{o,1} - \rho_{o,2}}_{\L1}
    \\
    & &
    +
    t \, e^{\kappa_o t}
    \norma{\direct_1 - \direct_2}_{\L\infty}
    \\
    &&
    \quad  \times
    \left[
      \tv(\rho_{o,1})
      +
      C_d \norma{\rho_1}_{\L\infty([0,T]\times\reali^d)}
      \int_0^t \int_{\reali^d}
      \norma{\nabla\div \, \direct_1(\tau, x)} \d{x} \d{\tau}
    \right]
    \\
    & &
    +
    \norma{\rho_2}_{\L\infty([0,T]\times\reali^d)}
    \int_0^t \int_{\reali^d}
    \modulo{\div \left(\direct_1(\tau, x) - \direct_2(\tau, x) \right)}
    \d{x} \d{\tau}
    \,.
  \end{eqnarray*}
\end{corollary}

\subsection{Proofs related to Section~\ref{sec:1}}
\label{subsec:TD1}

\begin{proofof}{Theorem~\ref{thm:main1}}
  \textbf{1.}  Let $\rho_o \in (\L1 \cap \L\infty \cap \BV) (\reali^d;
  \reali^n)$ and let $N_1=\norma{\rho_o}_{\L1}$. Let us introduce a
  given time $T>0$ and the sphere
  \begin{displaymath}
    \mathcal{X}_{N_1}
    =
    \left\{
      \rho \in \C0\left([0,T]; \L1(\reali^d; \reali^n) \right) \colon
      \textrm{ for all } t \in [0,T]\,,\;
      \norma{\rho(t)}_{\L1} \leq N_1
    \right\}\,,
  \end{displaymath}
  equipped with the distance induced by the norm $\norma{\, \cdot
    \,}_{\L\infty([0,T]; \L1(\reali^d; \reali^n))}$.  Consider first a
  fixed $i\in \{1, \ldots, n\}$. Choose any $r_1, r_2 \in
  \mathcal{X}_{N_1}$ and denote $\rho^i_1, \rho^i_2 \in
  \L\infty\left([0,T]; \L1(\reali^d; \reali)\right)$ the solutions of
  the Cauchy problems~(\ref{eq:scalar}) with respectively, for $k\in
  \{1, 2\}$,
  \begin{equation}
    \label{eq:choix1}
    q_k(\rho) = \rho
    \,,
    \quad
    \direct_k (t,x) = v^i(r^1_k * \eta^1 + \ldots + r^n_k * \eta^n)\,.
  \end{equation}
  Thanks to the properties of $v^i\in \C2$ and $\eta$ the hypotheses
  on $\direct_k$ in Lemma~\ref{lem:scalar} are satisfied. Hence, we
  get existence and uniqueness of a weak entropy solution, Besides,
  the $\L1$ bound on the solution
  $\norma{\rho(t)}_{\L1}=\norma{\rho_o}_{\L1}$ for any $t\geq 0$
  ensures that the solution of~(\ref{eq:scalar})-(\ref{eq:choix1})
  belongs to $\mathcal{X}_{N_1}$.

  Noting furthermore that
  \begin{eqnarray*}
    \direct^i & \in &
    \L\infty([0,T]\times \reali^d; \reali^d)\,,
    \\
    \nabla \direct^i & \in &
    \L\infty ([0,T]\times \reali^d; \reali^{d\times d})
    \cap
    \L\infty \left([0,T]; \L1(\reali^d; \reali^{d\times d}) \right)\,,
    \\
    \nabla^2 \direct^i & \in &
    \L\infty \left([0,T]; \L1(\reali^d; \reali^{d\times d\times d})\right)
  \end{eqnarray*}
  we can apply~\cite[Corollary~5.2 and Lemma~5.3]{ColomboHertyMercier}
  or more directly~\cite[Theorem~2.2 and Theorem~2.6]{Lecureux}.

  Remark that $N_1\geq \sup_{t\in [0,T]}\norma{ r(t)}_{\L1}$, so that
  we have the estimate
  \begin{eqnarray*}
    \norma{\div \direct^i(t)}_{\L\infty}
    &\leq &
    \norma{(v^i)'}_{\L\infty} \norma{r(t)}_{\L1} \norma{\nabla\eta}_{\L\infty}
    \norma{\vec v^i}_{\L\infty}
    +
    \norma{v^i}_{\L\infty} \norma{\nabla \vec v^i}_{\L\infty}
    \\
    &\leq &
    \norma{v^i}_{\W1\infty} \norma{\vec v^i}_{\W1\infty}
    (N_1 \norma{\nabla\eta}_{\L\infty} +1) \,.
  \end{eqnarray*}
  Thus, with
  \begin{equation}
    \label{eq:K1}
    K_1
    =
    \norma{v}_{\W1\infty}\norma{\vec  v}_{\W1\infty}
    (N_1 \norma{\nabla\eta}_{\L\infty} + 1)\,,
  \end{equation}
  we obtain by~\cite[Corollary~5.2]{ColomboHertyMercier}
  \begin{displaymath}
    \norma{\rho^i(t)}_{\L\infty}
    \leq
    \norma{\rho_o^i}_{\L\infty} e^{K_1 t}\,.
  \end{displaymath}
  Summing over $i = 1, \ldots, n$, we obtain
  \begin{displaymath}
    \norma{\rho(t)}_{\L\infty}
    \leq
    \norma{\rho_o}_{\L\infty}e^{K_1 t}\,.
  \end{displaymath}
  We want now to apply Lemma~\ref{lem:scalar}
  or~\cite[Theorem~2.2]{Lecureux}. Note that, with the notations
  of~\cite{Lecureux}, we have
  \begin{eqnarray*}
    \kappa_o
    & = &
    (2d+1) \norma{\nabla \direct^i}_{\L\infty}\leq (2d+1)K_1\,,
    \\
    \norma{\nabla \div \direct^i}_{\L1}
    &\leq &
    \norma{(v^i)''}_{\L\infty}
    \left(
      \sum_{j=1}^n \norma{\rho^j}_{\L1} \norma{\nabla\eta^j}_{\L\infty}
    \right)^2
    \norma{\vec v^i}_{\L1}
    \\
    & &
    +
    \norma{(v^i)'}_{\L\infty}
    \sum_j \norma{\rho^j}_{\L1} \norma{\nabla^2 \eta^j}_{\L\infty}
    \norma{ \vec v^i}_{\L1}
    \\
    & &
    +
    2\norma{(v^i)'}_{\L\infty} \sum_j \norma{\rho^j}_{\L1}
    \norma{\nabla \eta^j}_{\L\infty} \norma{\nabla \vec v^i}_{\L1}
    +
    \norma{v^i}_{\L\infty}\norma{\nabla^2 \vec v^i}_{\L1}
    \\
    & \leq &
    \norma{v^i}_{\W2\infty} \norma{\vec v^i}_{\W21}
    \left(
      N_1^2 \norma{\nabla\eta}_{\L\infty}^2
      +
      2N_1\norma{\nabla\eta}_{\W1\infty} +1
    \right)\,.
  \end{eqnarray*}
  Denoting
  \begin{equation}
    \label{eq:K2}
    K_2
    =
    \norma{v}_{\W1\infty}  \, \norma{\vec v}_{\W11}
    \left(
      N_1^2 \, \norma{\nabla\eta}_{\L\infty}^2
      +
      2 \, N_1 \, \norma{\nabla\eta}_{\W1\infty}
      +
      1
    \right)\,,
  \end{equation}
  we obtain
  \begin{displaymath}
    \tv\left( \rho^i_k(t) \right)
    \leq
    \tv( \rho^i_o ) e^{(2d+1)K_1 t}
    +
    t e^{K_1 t} d W_d K_2 \norma{\rho^i_o}_{\L\infty}\,.
  \end{displaymath}
  Let us now apply Corollary~\ref{cor:scalar}
  or~\cite[Theorem~2.6]{Lecureux}. We obtain
  \begin{displaymath}
    \norma{\rho^i_1-\rho^i_2}_{\L1}
    \leq
    K t e^{Kt} \norma{r_1 - r_2}_{\L\infty([0,T]; \L1)}
    \left[
      \tv(\rho^i_o)
      +
      (t\,d W_d K_2 + 2+N_1 \norma{\nabla \eta}_{\L\infty})
      \norma{\rho^i_o}_{\L\infty}
    \right]
  \end{displaymath}
  where
  \begin{displaymath}
    K
    =
    \norma{v'}_{\W1\infty} \,
    \norma{\eta}_{\W1\infty} \,
    (\norma{\vec v}_{\L\infty}+\norma{\vec v}_{\W11})\,.
  \end{displaymath}
  Note that $K_1$, $K_2$ and $K$ are constants depending on
  $\norma{v}_{\W2\infty}$, $\norma{\vec v}_{\L\infty}$, $\norma{\vec
    v}_{\W21}$, $\norma{\eta}_{\W2\infty}$, $N_1$ and $d$ and not on
  the initial condition $\rho_o$.  We do not need here precision on
  the constant so, with a $C$ large enough, not depending on $\rho_o$
  and $T$, we have, denoting $F(t) = C \, t \, e^{ C t}
  \left(\tv(\rho_o) +C(t+1) \, \norma{\rho_o}_{\L\infty} \right)$ and
  summing over $i = 1, \ldots, n$,
  \begin{eqnarray*}
    \norma{\rho(t)}_{\L\infty}
    & \leq &
    \norma{\rho_o}_{\L\infty}e^{Ct}\,,
    \\
    \tv \left( \rho(t) \right)
    & \leq &
    \tv(\rho_o) e^{Ct}
    +
    C t e^{Ct} \norma{\rho_o}_{\L\infty}\,,
    \\
    \norma{(\rho_1 - \rho_2)(t)}_{\L1}
    & \leq &
    F(t) \, \norma{r_1 - r_2}_{\L\infty([0,T]; \L1)}\,.
  \end{eqnarray*}
  Choose $T$ so that $F(T)=1/2$. Then, the map $\mathcal{Q}\colon
  \mathcal{X}_{N_1} \to \mathcal{X}_{N_1}$, defined by $\mathcal{Q}
  (r) = \rho$, is a contraction. Banach Fixed Point Theorem ensures
  local in time existence and uniqueness of weak entropy solutions
  to~\eqref{eq:SCL}--\eqref{eq:2}.

  In order to get global in time existence, we iterate the previous
  procedure. Starting from time $T_n$, we obtain
  \begin{eqnarray*}
    \norma{(\rho_1 - \rho_2) (t)}_{\L1}
    & \leq &
    \norma{r_1 - r_2}_{\L\infty([0,t]; \L1)} C (t-T_n) e^{C (t-T_n)}
    \\
    & &
    \times
    \left(
      \tv( \rho_o ) e^{CT_n} +CT_n e^{CT_n} \norma{\rho_o}_{\L\infty}
      +
      C \left( 1 +  t-T_n\right) e^{C T_n} \norma{\rho_o}_{\L\infty}
    \right).
  \end{eqnarray*}
  Iteratively, we choose $T_{n+1}>T_n$ such that
  \begin{displaymath}
    C\, (T_{n+1}-T_n) \, e^{CT_{n+1}} \,
    \left(
      \tv( \rho_o )
      +
      C \left(1 +  T_{n+1} )\right)  \,
      \norma{\rho_o}_{\L\infty}
    \right)
    =
    \frac{1}{2}\,.
  \end{displaymath}
  Since the sequence $T_n$ grows to $+\infty$, we have global in time
  existence, proving point~1.

  \textbf{2.}  The positivity of the solution directly follows from
  the fact that $\rho^i \equiv 0$ is a solution and from Kru\v zkov
  Maximum Principle~\cite[Theorem~3]{Kruzkov}.

  \textbf{3.} To prove these estimates, we
  apply~\cite[Corollary~5.2]{ColomboHertyMercier}
  and~\cite[Theorem~2.2 and Theorem~2.6]{Lecureux} to each equation of
  the system, which is possible since the coupling is only present in
  the nonlocal term. Let $i\in \{1, \ldots, n\}$. Denote $K_1$ as
  in~(\ref{eq:K1}), $N_1=\norma{\rho_o}_{\L1}=\sum_{j=1}^n
  \norma{\rho_o^j}_{\L1}$ and $\direct^i= v^i(\rho^1 * \eta^1 +
  \ldots + \rho^n * \eta^n) \, \vec v^i(x)$, we have
  \begin{displaymath}
    \norma{\rho(t)}_{\L\infty}
    \leq
    \norma{\rho_o}_{\L\infty}e^{K_1 t}\,.
  \end{displaymath}

  To prove the estimate on the total variation,
  apply~\cite[Theorem~2.2]{Lecureux}.  Denoting $K_2$ as
  in~(\ref{eq:K2}), we obtain
  \begin{displaymath}
    \tv \left( \rho^i(t) \right)
    \leq
    e^{(2d+1)K_1 t} \tv(\rho_o^i)
    +
    t e^{(2d+1)K_1 t} d W_d K_2 \norma{\rho_o^i}_{\L\infty}\,.
  \end{displaymath}
  Summing over $i = 1, \ldots, n$ we obtain the desired bound on $\tv
  \left( \rho(t) \right) = \sum_{i=1}^n \tv \left( \rho^i(t) \right)$.

  \textbf{4.} To prove the $\L1$ stability of the solution with
  respect to initial conditions and parameters, we
  apply~\cite[Theorem~2.10]{Lecureux}. Denote by $\rho_k^i$, for $k=1,
  2$, the solutions to the Cauchy problems
  \begin{displaymath}
    \left\{
      \begin{array}{l}
        \partial_t \rho^i_k +\div \left(\rho^i_k \, \direct^i_k (t,x)\right) = 0\,,
        \\
        \rho^i_k(0, \, \cdot \, ) = \rho^i_{o,k} \,.
      \end{array}
    \right.
    \quad \mbox{ where } \quad
    \direct^i_k(t,x)
    =
    v_k^i(\rho^1_k * \eta_k^1  + \ldots + \rho^n_k * \eta_k^n)
    \, \vec v_k^i(x)\,.
  \end{displaymath}
  Note that $\div \direct^i_k \in \L\infty \left([0,T]; \L1(\reali^d;
    \reali^d) \right)$ for $k=1, 2$, so that the necessary hypotheses
  hold. Moreover,
  \begin{eqnarray*}
    & &
    \norma{\div (\direct^i_1-\direct^i_2)(t)}_{\L1}
    \\
    & \leq &
    \norma{(v^i_1)'-(v^i_2)'}_{\L\infty}\norma{\vec v^i_1}_{\L1}
    \norma{\rho_{o,1}}_{\L1}
    \norma{\nabla\eta_1}_{\L\infty}
    +
    \norma{v^i_1-v^i_2}_{\L\infty} \norma{\div \vec v^i_1}_{\L1}
    \\
    & &
    +
    \left(
      \norma{(\rho_1-\rho_2)(t)}_{\L1} \norma{\eta_1}_{\L\infty}
      +
      \norma{\rho_{o,2}}_{\L1} \norma{\eta_1-\eta_2}_{\L\infty}
    \right)
    \\
    & &
    \qquad
    \times
    \left[   \norma{\vec v^i_1}_{\L1} \norma{(v^i_1)''}_{\L\infty}
      \norma{\rho_{o,1}}_{\L1}
      \norma{\nabla\eta_1}_{\L\infty}
      +
      \norma{\div\vec v^i_1}_{\L1} \norma{(v^i_1)'}_{\L\infty}
    \right]
    \\
    & &
    +
    \left(
      \norma{(\rho_1-\rho_2)(t)}_{\L1}
      \norma{\nabla\eta_1}_{\L\infty}
      +
      \norma{\rho_{o,2}}_{\L1}
      \norma{\nabla\eta_1 - \nabla\eta_2}_{\L\infty}
    \right)   \norma{\vec v^i_1}_{\L1} \norma{(v^i_2)'}_{\L\infty}
    \\
    & &
    +
    \norma{\div (\vec v^i_1-\vec v^i_2)}_{\L1} \norma{v^i_2}_{\L\infty}
    + \norma{\vec v^i_1-\vec v^i_2}_{\L1}
    \norma{(v^i_2)'}_{\L\infty} \norma{\rho_{o,2}}_{\L1}
    \norma{\nabla\eta_2}_{\L\infty}  \,,
  \end{eqnarray*}
  and
  \begin{eqnarray*}
    & &
    \norma{(\direct^i_1-\direct^i_2)(t)}_{\L\infty}
    \\
    & \leq &
    \norma{v^i_1-v^i_2}_{\L\infty} \norma{\vec v^i_1}_{\L\infty}
    +
    \norma{v^i_2}_{\L\infty} \norma{\vec v^i_1-\vec v^i_2}_{\L\infty}
    \\
    & &
    +
    \left(
      \norma{(\rho_1-\rho_2)(t)}_{\L1}
      \norma{\eta_1}_{\L\infty}
      +
      \norma{\rho_{o,2}}_{\L1}
      \norma{\eta_1-\eta_2}_{\L\infty}
    \right)
    \norma{(v^i_1)'}_{\L\infty}\norma{\vec v^i_1}_{\L\infty}\,.
  \end{eqnarray*}
  Let us introduce
  \begin{eqnarray*}
    \alpha
    & = &
    \norma{v_1''}_{\L\infty} \norma{\eta_1}_{\L\infty} \norma{\rho_{o,1}}_{\L1}
    \norma{\nabla \eta_1}_{\L\infty} \norma{\vec v_1}_{\L1}
    +
    \norma{v_2'}_{\L\infty} \norma{\nabla\eta_1}_{\L\infty} \norma{\vec v_1}_{\L1}
    \\
    & &
    +
    \norma{v_1'}_{\L\infty} \norma{\eta_1}_{\L\infty} \norma{\div \vec v_1}_{\L1}\,,
    \\
    \beta
    & = &
    \norma{v_1''}_{\L\infty} \norma{\rho_{o,2}}_{\L1} \norma{\rho_{o,1}}_{\L1}
    \norma{\nabla \eta_1}_{\L\infty} \norma{\vec v_1}_{\L1}
    +
    \norma{v_1'}_{\L1} \norma{\rho_{o,2}}_{\L1} \norma{\div \vec v_1}_{\L1}
    \\
    & &
    +
    \norma{v_2'}_{\L\infty} \norma{\rho_{o,2}}_{\L1} \norma{\vec v_1}_{\L1}\,,
    \\
    \gamma
    & = &
    \norma{\div \vec v_1}_{\L1}
    +
    \norma{\rho_{o,1}}_{\L1} \norma{\nabla \eta_1}_{\L\infty}
    \norma{\vec v_1}_{\L1}\,,
    \\
    \delta
    & = &
    \norma{v_2'}_{\L\infty} \norma{\rho_{o,2}}_{\L1} \norma{\nabla\eta_2}_{\L\infty}
    +
    \norma{v_2}_{\L\infty}\,,
  \end{eqnarray*}
  and
  \begin{displaymath}
    \alpha' = \norma{v_1'}_{\L\infty}\norma{\eta_1}_{\L\infty}
    \,,\quad
    \beta' = \norma{v_1'}_{\L\infty} \norma{\rho_{o,2}}_{\L1}
    \norma{\vec v_1}_{\L\infty}
    \,,\quad
    \gamma' = \norma{\vec v_1}_{\L\infty}
    \,,\quad
    \delta' = \norma{v_2}_{\L\infty}\,,
  \end{displaymath}
  these coefficients being chosen so that
  \begin{eqnarray*}
    \norma{\div (\direct^i_1-\direct^i_2)(\tau)}_{\L1}
    &\leq &
    \alpha \norma{(\rho_1-\rho_2)(\tau)}_{\L1}
    +
    \beta \norma{\eta_1-\eta_2}_{\W1\infty}
    +
    \gamma \norma{v_1-v_2}_{\W1\infty}
    \\
    & &
    +
    \delta \norma{\vec v_1-\vec v_2}_{\W11}\,,
    \\
    \norma{(\direct^i_1-\direct^i_2)(\tau)}_{\L\infty}
    & \leq &
    \alpha' \norma{(\rho_1-\rho_2)(\tau)}_{\L1}
    +
    \beta' \norma{\eta_1-\eta_2}_{\L\infty}
    +
    \gamma' \norma{v_1-v_2}_{\L\infty}
    \\
    & &
    +
    \delta' \norma{\vec v_1-\vec v_2}_{\L\infty} \,.
  \end{eqnarray*}
  Hence, we get
  \begin{eqnarray*}
    & &
    \norma{(\rho_1^i-\rho_2^i)(t)}_{\L1}
    \\
    & \leq &
    \norma{\rho_{o,1}^i-\rho_{o,2}^i}_{\L1}
    +
    e^{(2d+1)K_1 t}
    \left[
      \tv(\rho_{o,1})
      +
      t dW_d K_2 \norma{\rho_{o,1}}_{\L\infty}
    \right]
    \int_0^t \norma{(\direct^i_1-\direct^i_2)(\tau)}_{\L\infty} \d{\tau}
    \\
    & &
    +
    \max
    \left\{
      \norma{\rho_{1}(t)}_{\L\infty},
      \norma{\rho_{2}(t)}_{\L\infty}
    \right\}
    \int_0^t
    \norma{\div (\direct^i_1-\direct^i_2)(\tau)}_{\L1} \d{\tau}
    \\
    &\leq &
    \norma{\rho_{o,1}^i-\rho_{o,2}^i}_{\L1}
    +
    A_\rho(t) \int_0^t \norma{(\rho_1^i-\rho_2^i)(\tau)}_{\L1}\d{\tau}
    +
    t \, A_v(t)\norma{v_1-v_2}_{\W1\infty}
    \\
    & &
    +
    t \, A_\eta(t) \norma{\eta_1-\eta_2}_{\W1\infty}
    +
    t \, A_{\vec v}(t)
    \left(
      \norma{\vec v_1-\vec v_2}_{\W11}
      +
      \norma{\vec v_1-\vec v_2}_{\L\infty}
    \right) \,,
  \end{eqnarray*}
  where $A_\rho$, $A_v$, $A_\eta$, $A_{\vec v}$ are smooth, positive,
  and increasing, functions of $t$ depending on $d$,
  $\norma{v}_{\W2\infty}$, $\norma{\vec v}_{\W11}$, $\norma{\vec
    v}_{\L\infty}$, $\norma{\eta}_{\W1\infty}$, $\tv (\rho_{o,1})$,
  $\max \left\{ \norma{\rho_{o,1}}_{\L\infty},
    \norma{\rho_{o,2}}_{\L\infty} \right\}$,
  $\norma{\rho_{o,1}}_{\L1}$ and $\norma{\rho_{o,2}}_{\L1}$.  More
  precisely, denoting
  \begin{eqnarray*}
    f(t)
    & = &
    e^{(2d +1)K_1}\left(\tv(\rho_{o,1})
      +
      t d W_d K_2 \norma{\rho_{o,1}}_{\L\infty}\right)\,,
    \\
    R
    & = &
    \max (\norma{\rho_{o,1}}_{\L\infty},\norma{\rho_{o,2}}_{\L\infty})\,,
    \\
    K
    & = &
    \max_{k\in \{ 1, 2\} } \left\{
      \norma{v_k}_{\W1\infty} \norma{\vec v_k}_{\W1\infty}
      ( \norma{\rho_{o,k}}_{\L1} \norma{\nabla\eta_k}_{\L\infty}  +1)
    \right\} \,,
    \\
  \end{eqnarray*}
  we have
  \begin{displaymath}
    \begin{array}{rcl@{\qquad}rcl}
      A_\rho(t) & = & \alpha' f(t)+\alpha Re^{Kt}\,, &
      A_\eta(t) & = & \beta' f(t)+\beta Re^{Kt}\,,
      \\
      A_v(t) & = & \gamma' f(t)+\gamma Re^{Kt}\,,&
      A_{\vec v}(t) & = & \delta' f(t)+\delta Re^{Kt}\,.
    \end{array}
  \end{displaymath}
  We conclude applying the Gronwall Lemma, obtaining
  \begin{eqnarray*}
    \norma{(\rho_1-\rho_2)(t)}_{\L1}
    & \leq &
    \left(
      \norma{\rho_{o,1}-\rho_{o,2}}_{\L1}
      +
      t A_\eta(t) \norma{\eta_1-\eta_2}_{\W1\infty}
      +
      t A_{v}(t) \norma{v_1-v_2}_{\W1\infty} \right.
    \\
    & &
    \left.
      +
      t A_{\vec v}(t)
      \left(
        \norma{\vec v_1-\vec v_2}_{\L\infty}
        +
        \norma{\vec v_1-\vec v_2}_{\W11}
      \right)
    \right)
    \\
    & &
    \qquad \times
    \left(
      1
      +
      t \left(f(t) \alpha' + R e^{Kt} \alpha \right)
      e^{tf(t)\alpha' + t e^{Kt} R\alpha }
    \right) \,,
  \end{eqnarray*}
  and the stability estimate at~4.~follows.

  \textbf{5.} These properties follow through standard computations
  from the representation formula given
  in~\cite[Lemma~5.1]{ColomboHertyMercier}.

  \textbf{6.} As above, we deduce this further regularity property and
  the $\W21$ estimate from the representation formula
  in~\cite[Lemma~5.1]{ColomboHertyMercier}.

  \textbf{7.} Let $\rho\in \C0 \left(\reali^+; \L1(\reali^d; \reali^n)
  \right)$ be a solution to the initial
  problem~(\ref{eq:SCLn})--(\ref{eq:2}) such that for all $t\geq 0$,
  $\rho(t)\in \W2\infty\cap\W21(\reali^d; \reali^n)$. Consider now the
  linearized equations
  \begin{displaymath}
    \left\{
      \begin{array}{l}
        \partial_t \sigma^1
        +
        \div
        \left(
          \left(
            \rho^1 \; (v^1)'(\rho \convn \eta) \; \sigma \convn \eta
            +
            \sigma^1 \; v^1 (\rho \convn \eta)
          \right)
          \vec v^1(x)
        \right)
        =
        0
        \,,
        \\
        \ldots
        \\
        \partial_t \sigma^n
        +
        \div
        \left(
          \left(
            \rho^n \; (v^n)'(\rho\convn \eta ) \; \sigma\convn\eta
            +
            \sigma^n \; v^n (\rho\convn \eta)
          \right)
          \vec v^n(x)
        \right)
        =
        0\,,
      \end{array}
    \right.
  \end{displaymath}
  where $\rho = (\rho^1, \ldots, \rho^n)$, $\eta = (\eta^1, \ldots,
  \eta^n)$, $\sigma = (\sigma^1, \ldots \sigma^n)$ and $ \rho \convn
  \eta = \rho^1 \conv \eta^1 + \ldots + \rho^n \conv \eta^n$.  To
  prove the existence and uniqueness of weak entropy solutions to this
  linearized problem, we use the technique that proved to be effective
  for the initial value problem: let $M_1=\norma{\sigma_o}_{\L1}$, let
  $s_1, s_2\in \mathcal{X}_{M_1}$ are fixed functions; we fix the
  nonlocal term and study the Cauchy problems
  \begin{displaymath}
    \!\left\{
      \begin{array}{@{\,}l@{}}
        \partial_t \sigma_k^i
        +
        \div
        \left(
          \sigma^i \; v^i (\rho\convn \eta)  \vec v^i(x)
        \right)
        =
        -\div \left(
          \rho^i \; (v^i)'(\rho\convn \eta) \; s_k\convn\eta \;\vec v^i(x)
        \right)
        \\
        \sigma_{o,k}^i = \sigma_o^i
      \end{array}
    \right.
    \mbox{ for }
    \left\{
      \begin{array}{@{\,}rcl@{}}
        k & = & 1,2
        \\
        i & = & 1, \ldots, n
      \end{array}
    \right.
  \end{displaymath}
  We study the map $\mathcal{T} \colon \mathcal{X}_{M_1} \to
  \mathcal{X}_{M_1} $ defined by $\mathcal{T}(s) = \sigma$. This
  application is well defined thanks to Kru\v zkov
  Theorem~\cite{Kruzkov} and thanks
  to~\cite[Lemma~5.1]{ColomboHertyMercier}. Indeed, denoting
  \begin{displaymath}
    \begin{array}{rcl@{\qquad}rcl}
      f(t,x,u) & = &
      u \, v^i(\rho\convn \eta) \,  \vec v^i(x)\,,
      &
      F(t,x,u) & = &
      -\div \left(\rho^i \, (v^i)'(\rho\convn\eta) \, s_1\convn\eta  \vec v^i(x)\right)\,,
      \\
      g(t,x,u) & = &
      u \, v^i(\rho\convn \eta) \,  \vec v^i(x)\,,
      &
      G(t,x,u)
      & = &
      -\div \left(\rho^i \, (v^i)'(\rho\convn\eta) \, s_2\convn\eta \vec v^i(x) \right)\,,
    \end{array}
  \end{displaymath}
  we have $\pt_u f, \pt_u g\in \L\infty([0,T]\times\reali^d\times
  [-U,U])$, $F-\div f, G - \div g \in
  \L\infty([0,T]\times\reali^d\times [-U,U])$, $\pt_u (F-\div f),
  \pt_u (G-\div g)\in \L\infty([0,T]\times\reali^d\times
  [-U,U])$. Hence Kru\v zkov hypotheses are satisfied. To
  apply~\cite[Theorem~2.6]{Lecureux}, we have now to check that
  \begin{displaymath}
    \begin{array}{rcl}
      \nabla \pt_u f  \in \L\infty([0,T]\times\reali^d\times [-U,U])\,,
      \\
      \int_0^T \int_{\reali^d}
      \norma{\nabla(F-\div f)}_{\L\infty([-U, U])} \d{x} \d{t}
      & < &
      +\infty\,,
      \\
      \int_0^T\int_{\reali^d}
      \norma{F-G-\div (f-g)}_{\L\infty([-U,U])} \d{x} \d{t}
      & < &
      +\infty\,.
    \end{array}
  \end{displaymath}
  We have
  \begin{eqnarray*}
    \norma{\nabla F(t,\cdot,u)}_{\L1}
    &\leq &
    \norma{\rho(t)}_{\W2\infty} \norma{v'}_{\W1\infty} \norma{s_1}_{\L1}
    \norma{\eta}_{\W2\infty} \norma{\vec v}_{\W21}
    \\
    & &
    \times
    \Bigl[
    9
    +
    6 \norma{\rho(t)}_{\L1} \norma{\nabla \eta}_{\L\infty}
    +
    \norma{\rho(t)}_{\L1}^2 \norma{\nabla \eta}_{\L\infty}^2
    \\
    & &
    \qquad\qquad
    +
    \norma{\rho(t)}_{\L1} \norma{\nabla^2 \eta}_{\L\infty}
    \Bigr] ,
    \\
    \norma{\nabla\div f(t, \cdot, u)}_{\L1}
    & \leq &
    \modulo{u} \, K_2\,,
    \\
    \norma{(F-G)(t,\cdot,u)}_{\L1}
    &\leq &
    \norma{\rho(t)}_{\W1\infty} \norma{v'}_{\W1\infty}
    \norma{\eta}_{\W1\infty} \norma{\vec v}_{\W11}
    \left(
      3
      +
      \norma{\rho(t)}_{\L1} \norma{\nabla\eta}_{\L\infty}
    \right)
    \\
    & &
    \times \norma{(s_1-s_2)(t)}_{\L1}\,,
    \\
    \norma{\div\!(f-g)(t,\cdot,u)}_{\L1} & = &0\,,
    \\
    \norma{\pt_u(f-g)(t,\cdot,u)}_{\L\infty} & = &0\,.
  \end{eqnarray*}
  where $K_2$ is defined as in~(\ref{eq:K2}). Using the estimate on
  $\norma{\rho(t)}_{\W1\infty}$ obtained in~4.~and denoting
  \begin{equation}
    \label{eq:K4}
    K_4
    =
    \norma{\rho_o}_{\W1\infty} \norma{v'}_{\W1\infty}
    \norma{\eta}_{\W1\infty} \norma{\vec v}_{\W11}
    \left(
      3
      +
      \norma{\rho_o}_{\L1} \norma{\nabla\eta}_{\L\infty}
    \right)\,,
  \end{equation}
  we obtain
  \begin{eqnarray*}
    \norma{\sigma_1^i (t) - \sigma^i_2(t)}_{\L1}
    & \leq &
    \int_0^t
    \norma{\div \left(
        \rho^i \, (v^i)'(\rho \convn \eta) \, (s_1-s_2) \convn \eta
        \,  \vec v^I
      \right)}_{\L1}
    \d{t}
    \\
    & \leq &
    t \, K_4 (1+Ct)e^{Ct} \, \norma{s_1-s_2}_{\L\infty([0,t], \L1)} \,.
  \end{eqnarray*}
  For $T$ small enough, we
  obtain that $\mathcal{T}$ is a contraction, proving the local in
  time existence and uniqueness of solutions by Banach Fixed Point
  Theorem. We then extend to $+\infty$ the time of existence by
  iteration of the process.

  Denote now by $\rho$, respectively $\rho_h$, the solution to the
  initial problem~\eqref{eq:SCLn}--\eqref{eq:2} with initial
  conditions $\rho_o$, respectively $\rho_o + h \,
  \sigma_o$. Moreover, call $\sigma$ the solution to the linearized
  equation~\eqref{eq:linear} with initial condition $\sigma_o$ and
  define $z_h = \rho + h \, \sigma$. If $\sigma $ is smooth enough we
  can write for $z_h$ the equation
  \begin{displaymath}
    \partial_t z^i_h
    +
    \div \left(
      z^i_h \left(
        v^i(\rho\convn \eta)
        +
        h (v^i)'(\rho\convn \eta) \, \sigma\convn\eta
      \right)
      \vec v^i(x)
    \right)
    =
    h^2 \div
    \left(
      \sigma^i\, (v^i)'(\rho\convn \eta) \, \sigma\convn\eta \vec v^i(x)
    \right)\,.
  \end{displaymath}
  We want now to estimate $\norma{(\rho_h-z_h)(t)}_{\L1}/h$. In order
  to do that, we use~\cite[Theorem~2.6]{Lecureux}.  As it is similar
  to the estimate in the proof of~\cite[Theorem~2.10]{ColomboHertyMercier}, we omit it.
\end{proofof}

\subsection{Proofs related to Section~\ref{sec:2}}
\label{subsec:TD2}

\begin{lemma}\label{lem:inv0R}
  In the same setting as Lemma~\ref{lem:scalar}, if furthermore for a
  given $R>0$, we have $q(R)=0$, then
  \begin{displaymath}
    \rho_o\in[0,R] \qquad \Rightarrow \qquad \forall t\geq 0\,,\quad
    \rho(t)\in [0,R]\,.
  \end{displaymath}
\end{lemma}

\begin{proof}
  Since $q(0)=q(R)=0$, $\rho\equiv 0$ and $\rho\equiv R$ are solutions
  to~(\ref{eq:scalar}). The maximum principle of
  Kru\v{z}kov~\cite{Kruzkov} then ensures that $0\leq \bar\rho \leq R $
  implies $0 \leq \rho(t) \leq R$ for all $t \geq 0$.
\end{proof}

\begin{proofof}{Theorem~\ref{thm:main2}}
  Consider the following steps separately.

  \paragraph{Existence.}
  Fix an arbitrary positive time $T$, whose precise value will be
  chosen later. Let $(r^1, \ldots , r^n), (s^1, \ldots, s^n) \in \C0
  \left([0,T]; \L1(\reali^d; [0,R]^n)\right)$. For any $i\in \{1,
  \ldots, n\}$, we define
  \begin{displaymath}
    \direct^i(t,x)
    =
    { \vec v^i}(x)
    +
    \mathcal{I}^i \left( r^1(t), \ldots, r^n(t)\right)(x)\,,
    \qquad
    W^i(t,x)
    =
    { \vec v^i}(x)
    +
    \mathcal{I}^i \left(s^1(t), \ldots, s^n(t) \right)(x)\,.
  \end{displaymath}
  As $v$ satisfies \textbf{($\mathbf{v}$)} and $\mathcal{I}$ satisfies
  $\mathbf{(I)}$, for any $r,s\in \C0\left([0,T], \L1(\reali^d, [0,R]^n)\right)$,
  we have for any $t\geq 0$, $\direct(t), W(t)\in
  (\C2\cap\W1\infty)(\reali^d, \reali^d)$ and $\div \direct(t), \div
  W(t)\in \W11(\reali^d, \reali)$.  Then, according to
  Lemma~\ref{lem:scalar}, there exists $\rho, \sigma \in
  \C0\left([0,T]; \L1(\reali^d; \reali_+^n)\right)$, that are the weak
  entropy solutions to the systems of decoupled equations
  \begin{displaymath}
    \left\{
      \begin{array}{l}
        \partial_t \rho^1+ \div
        \left( \rho^1 v^1(\rho^1) \, \direct^1 (t,x)\right) = 0\,,
        \\
        \ldots
        \\
        \partial_t \rho^n+ \div
        \left(\rho^n v^n(\rho^n) \, \direct^n (t,x)\right) = 0\,,
      \end{array}
    \right.
    \qquad
    \left\{
      \begin{array}{l}
        \partial_t \sigma^1+ \div
        \left(\sigma^1 v^1(\sigma^1) \, W^1(t,x)\right)=0\,,
        \\
        \ldots
        \\
        \partial_t \sigma^n+ \div
        \left(\sigma^n v^n(\sigma^n) \, W^n(t,x)\right)=0\,,
      \end{array}
    \right.
  \end{displaymath}
  with initial condition $\rho(0) = \sigma(0) = \rho_o \in
  (\L1\cap\L\infty)(\reali^d; [0,R]^n)$.

Note furthermore that $q^i(R)=0$;  thus, thanks to  Lemma \ref{lem:inv0R}, that  if $\rho$ is a solution of the above equation, then $\rho_o\in [0,R]$ implies for all $t\geq 0$, $\rho(t)\in [0,R]$. Hence, we have invariance of the interval $[0,R]$.

  Consider the map
  \begin{displaymath}
    \mathcal{T} \colon \left\{
      \begin{array}{ccc}
        \C0 \left([0,T]; \L1(\reali^d; [0,R]^n)\right)
        & \longrightarrow &
        \C0 \left([0,T]; \L1(\reali^d; [0,R]^n)\right)
        \\
        r
        &\mapsto
        &\rho
      \end{array}\right\}\,.
  \end{displaymath}
  Use Lemma~\ref{lem:scalar} and define $\kappa_o = \max_i
  \left\{(2d+1) \norma{(q^i)'}_{\L\infty([0,R])} \norma{\nabla
      \direct^i}_{\L\infty([0,T]\times\reali^d)}\right\}$, we get:
  \begin{eqnarray*}
    \norma{\rho^i(t)-\sigma^i(t)}_{\L1}
    & \leq &
    t\, e^{\kappa_o t} \norma{(q^i)'}_{\L\infty([0,R])}
    \norma{ \direct^i-W^i}_{\L\infty([0,t]\times \reali^d)}
    \\
    & &
    \qquad
    \times
    \left[
      \tv(\rho^i_o)
      +
      C_d \, \norma{q^i}_{\L\infty([0,R])} \int_0^t\int_{\reali^d}
      \norma{\nabla\div  \direct^i (\tau, x) } \d{x} \d{\tau}
    \right]
    \\
    & &
    +
    \norma{q^i}_{\L\infty([0,R])} \int_0^t \int_{\reali^d}
    \modulo{\div \left( \direct^i(\tau, x)-W^i(\tau, x) \right)} \d{x} \d{\tau}\,.
  \end{eqnarray*}
  Using hypothesis \textbf{(I)}, we obtain
  \begin{eqnarray*}
    & &
    \norma{\rho^i(t)-\sigma^i(t)}_{\L1}
    \\
    &\leq &
    t\, C_I \, \norma{r-s}_{\L\infty([0,t]; \L1( \reali^d; \reali^n))}
    \\
    & &
    \times
    \Big[
    t \, e^{\kappa_o t} \norma{(q^i)'}_{\L\infty([0,R])}
    \, C_d \, \norma{q^i}_{\L\infty([0,R])}
    \left(
      \norma{\nabla \div { \vec v^i}(x)}_{\L1}
      +
      C_I \, \norma{r}_{\L\infty([0,t], \L1(\reali^d; \reali^n))}
    \right)
    \\
    & &
    +
    e^{\kappa_o t} \, \norma{(q^i)'}_{\L\infty([0,R])} \tv(\rho^i_o)
    +
    \norma{q^i}_{\L\infty([0,R])}
    \Big] \,.
  \end{eqnarray*}
  Hence, for $T$ small enough, we can apply the Banach Fixed Point
  Theorem.

  Now, using the total variation estimate~\eqref{eq:bv}, a standard
  iteration procedure allows to obtain global in time existence.

  \paragraph{Total Variation Estimate.}
  To prove the estimate on the total variation, apply
  Lemma~\ref{lem:scalar} and use a procedure entirely similar to that
  exploited in the proof of Theorem~\ref{thm:main1}.

  \paragraph{Stability.}
  We use~\cite[Proposition~2.10]{Lecureux} to obtain
  \begin{eqnarray*}
    & &
    \norma{\rho_1(t) - \rho_2(t)}_{\L1}
    \\
    &\leq &
    \norma{\rho_{o,1}-\rho_{o,2}}_{\L1}
    +
    e^{\kappa_o t}
    \left(
      \tv(\rho_{o,1})
      +
      t \,C_d \, \norma{q_1}_{\L\infty} (C_I+\norma{\nabla\div \vec v_1}_{\L1} )
    \right)
    \\
    & &
    \qquad
    \times
    \Bigl(
    \norma{q_1'}_{\L\infty([0,R])} \int_0^t
    \norma{
      \mathcal{I} \left(\rho_1(\tau) \right)
      -
      \mathcal{I} \left(\rho_2(\tau) \right)
    }_{\L\infty(\reali^d)}
    \d{\tau}
    \\
    & &
    \qquad\qquad
    +
    t \norma{q_1'}_{\L\infty([0,R])} \norma{\vec v_2-\vec v_1}_{\L\infty}
    +
    t(C_I+\norma{\vec v_2}_{\L\infty}) \norma{q_1'-q_2'}_{\L\infty([0,R])}
    \Bigr)
    \\
    & &
    +
    \norma{q_1}_{\L\infty([0,R])} \int_0^t  \norma{
      \div\mathcal{I} \left(\rho_1(\tau) \right)
      -
      \div\mathcal{I} \left(\rho_2(\tau) \right)}_{\L1(\reali^d)}
    \d{\tau}
    \\
    & &
    +
    t \norma{q_1}_{\L\infty([0,R])} \norma{\div \vec v_2-\div \vec v_1}_{\L1}
    +
    t (C_I+\norma{\div\vec v_2}_{\L1})\norma{q_1-q_2}_{\L\infty([0,R])}
    \\
    &\leq &
    \norma{\rho_{o,1}-\rho_{o,2}}_{\L1}
    \\
    & &
    +
    t
    \Bigl[
    (C_I+\norma{\vec v_2}_{\L\infty})  e^{\kappa_o t}
    \\
    & &
    \qquad\qquad
    \times
    \left(
      \tv(\rho_{o,1})
      +
      t C_d \,  \norma{q_1}_{\L\infty([0,R])}
      (C_I+\norma{\nabla\div \vec v_1}_{\L1})
    \right)
    \norma{q_1'-q_2'}_{\L\infty([0,R])}
    \\
    & &
    \qquad\quad
    +
    (C_I+\norma{\div\vec v_2}_{\L1}) \norma{q_1-q_2}_{\L\infty([0,R])}
    \Bigr]
    \\
    & &
    +
    t \Bigl[
    e^{\kappa_o t}\norma{q_1'}_{\L\infty([0,R])}
    \left(
      \tv(\rho_{o,1})
      +
      t \, C_d \norma{q_1}_{\L\infty([0,R])} (C_I+\norma{\nabla\div \vec v_1}_{\L1} )
    \right)
    \norma{\vec v_2-\vec v_1}_{\L\infty}
    \\
    & &
    \qquad\quad
    +
    \norma{q_1}_{\L\infty([0,R])} \norma{\div(\vec v_2-\vec v_1)}_{\L1}
    \Bigr]
    \\
    & &
    +
    C_I \Bigl[
    e^{\kappa_o t}  \norma{q_1'}_{\L\infty([0,R])}
    \left(
      \tv(\rho_{o,1})
      +
      t C_d \norma{q_1}_{\L\infty([0,R])} (C_I+\norma{\nabla\div\vec v_1}_{\L1})
    \right)
    \\
    & &
    \qquad\quad
    +
    \norma{q_1}_{\L\infty([0,R])}
    \Bigr]
    \\
    & &
    \qquad
    \times
    \int_0^t \norma{\rho_1(\tau)-\rho_2(\tau)}_{\L1} \, \d{\tau}\,.
  \end{eqnarray*}
  Let us denote
  \begin{eqnarray*}
    a(t)
    & = &
    t \Bigl[
    (C_I+\norma{\vec v_2}_{\L\infty})  e^{\kappa_o t}
    \left(
      \tv(\rho_{o,1}) + t \, C_d  \norma{q_1}_{\L\infty([0,R])} \, (C_I+\norma{\nabla\div \vec v_1}_{\L1})
    \right)
    \\
    & &
    \qquad
    \times
    \norma{q_1'-q_2'}_{\L\infty([0,R])}
    \\
    & &
    +
    (C_I+\norma{\div\vec v_2}_{\L1})
    \norma{q_1-q_2}_{\L\infty([0,R])}
    +
    \norma{q_1}_{\L\infty([0,R])} \norma{\div(\vec v_2-\vec v_1)}_{\L1}
    \\
    & &
    +
    e^{\kappa_o t} \norma{q_1'}_{\L\infty([0,R])}
    \left(
      \tv(\rho_{o,1})
      +
      t \,C_d  \norma{q_1}_{\L\infty([0,R])} \, (C_I+\norma{\nabla\div \vec v_1}_{\L1} )
    \right)
    \\
    & &
    \qquad
    \times
    \norma{\vec v_2-\vec v_1}_{\L\infty}
    \Bigr]\,,
    \\
    b(t)
    & = &
    C_I
    \bigl[
    e^{\kappa_o t}  \norma{q_1'}_{\L\infty([0,R])}
    \left(
      \tv(\rho_{o,1})
      +
      t \, C_d  \norma{q_1}_{\L\infty([0,R])} (C_I+\norma{\nabla\div\vec v_1}_{\L1})
    \right)
    \\
    & &
    \qquad
    +
    \norma{q_1}_{\L\infty([0,R])}
    \bigr]\,.
  \end{eqnarray*}
  so that, by integration, we get
  \begin{displaymath}
    \norma{\rho_1(t)-\rho_2(t)}_{\L1}
    \leq
    \left(1+t \, e^{t\,b(t)} \right)
    \left( \norma{\rho_{o,1}-\rho_{o,2}}_{\L1}  + a(t) \right)\,,
  \end{displaymath}
  which completes the proof.
\end{proofof}

\noindent\textbf{Acknowledgment:\quad} The second author was partially
supported by the GNAMPA~2011 project \emph{Non Standard Applications
  of Conservation Laws}.

\small{

  \bibliography{ColomboLecureux}

  \bibliographystyle{abbrv}

}
\end{document}